\documentclass{article}
\usepackage[utf8]{inputenc}
\usepackage{amsmath, amssymb, graphicx, stmaryrd, amsthm, mathtools}
\usepackage{microtype}
\usepackage{cancel} %for slashing out E_d and stuff \bcancel{E_d}
\usepackage{calc} % this understands \textwidth
\usepackage{tikz}
\usepackage{microtype} % for pretty typesetting

\newcommand{\texorpdfstring[2]}{#1}
\usepackage{verbatim} % understands \begin{comment} \end{comment}
\usepackage{tabularx} %for column vertical centering in tables with pictures
\usepackage{xparse}
\usetikzlibrary{positioning}%positioning for "on grid" and "above= of this node"
\usetikzlibrary{matrix}%matrix for matrices

\tikzset{math nodes/.style={execute at begin node=$,execute at end node=$}}
\tikzset{display math nodes/.style={execute at begin node=$\displaystyle, execute at end node=$}}

\newcommand{\calO}{\mathcal{O}}
\newcommand{\calP}{\mathcal{P}}
\newcommand{\calQ}{\mathcal{Q}}

\newcommand{\aut}{\mathrm{aut}}
\newcommand{\dom}{\mathrm{dom}\,}

\newcommand{\col}{\mathrm{col}}
\newcommand{\rme}{\mathrm{e}}

\newcommand{\rmE}{\mathrm{E}}
\newcommand{\edge}{\rme}
\newcommand{\Edge}{\rmE}

\newcommand{\Exedge}{\mathrm{Ex}}

\newcommand{\rmf}{\mathrm{f}}
\newcommand{\f}{\rmf}

\newcommand{\fin}{\mathrm{Fin}}
\newcommand{\leaf}{\mathrm{leaf}}
\newcommand{\rt}{\mathrm{rt}}

\newcommand{\unit}{\mathrm{un}}

\newcommand{\rmun}{\mathrm{un}}
\NewDocumentCommand{\Op}{gg}% g An optional argument given inside a pair of TeX group tokens {...}, which returns -NoValue- if not present
  {%
  \IfNoValueTF{#2}%
    {\IfNoValueTF{#1}%
      {\mathrm{Op}}%
      {\mathrm{Op}_{#1}}}%
    {\mathrm{Op}_{#1}^{#2}}%
  }%
\DeclareUnicodeCharacter{0131}{\i}

\newcommand{\map}{\mathrm{map}}

\newcommand{\h}{\mathrm{h}}

\newcommand{\e}{\mathsf{f}}

\newcommand{\id}{\mathsf{id}}

\newcommand{\topl}{\mathsf{Top}}

\newcommand{\Coll}{\mathsf{Coll}}
\newcommand{\Collleqone}{{\Coll^{\leq 1}}}
\newcommand{\reals}{\mathbb{R}}
\newcommand{\Operads}{\Op{}{}}

\newcommand{\inv}{^{-1}}

\newcommand{\abs}[1]{\left| #1 \right|}

\newcommand{\modsim}{/\!\!\sim\,}

\newcommand{\smallbullet}{{\scriptstyle\bullet}}
\newcommand{\W}{W}

\newcommand{\swcheese}{{SC}_{\!d}}
\newcommand{\swcheesebullet}{{SC}_{\!d,\smallbullet}}
\newcommand{\wswcheese}{\mathsf{SC}_{\mathsf{d}}}

\newcommand{\schone}{\mathsf{SC}_{\mathsf d}^{\h 1}}
\newcommand{\sch}{{\wswcheese^{\h}}}
\newcommand{\schonebullet}{\mathsf{SC}_{\mathsf d,\smallbullet}^{\h 1}}
\newcommand{\schbullet}{{\mathsf{SC}_{\mathsf d, \smallbullet}^{\h}}}

\newcommand{\swisscheese}{\mathsf{SC}}
\newcommand{\schinf}{{\mathsf{SC}_{\mathsf{d}}^{\h\infty}}}

\newcommand{\rfor}{\mathsf{For}}
\newcommand{\for}{\mathsf{For}}
\newcommand{\trees}{\mathsf{Trees}}

\newcommand{\inverts}{{\mbox{\i}\mathrm{n}}}
\newcommand{\outverts}{\mathrm{out}}
\newcommand{\edges}{\textsl{Edge}}

\tikzset{map name/.style={font=\scriptsize}}

\tikzset{numbered picture/.style={baseline=(current bounding box.center)}}

\DeclareMathOperator*{\colim}{colim}

\DeclareMathOperator*{\hocolim}{hocolim}

\theoremstyle{definition}
 \newtheorem*{lemma*}{Lemma}                                        %%
 \newtheorem*{theorem*}{Theorem}
\newtheorem{theorem}[equation]{Theorem}
\newtheorem{corollary}[equation]{Corollary}
\newtheorem{lemma}[equation]{Lemma}                                %%
\newtheorem{proposition}[equation]{Proposition}
\newtheorem{remark}[equation]{Remark}                   %%

\newtheorem{definition}[equation]{Definition}

\title{Forests and the W construction}
\author{Justin D. Thomas}
\begin{document}

\maketitle

\begin{abstract}
	A modified definition of the category of forests found in Costello \cite{costello-a-infinity} and Getzler \cite{getzler-operads-revisited} is given.  This language is used to prove some cofibration properties of the $W$ construction for colored operads.  An application is given to the swiss cheese operad.  The language of forests is convenient for showing that the swiss cheese operad is equivalent to its free degree 0 and 1 pieces.
\end{abstract}

\section{The category of Forests}
\label{section:forests}
\begin{definition}
  \label{definition:K-colored-sets}
  Let $\fin$ be the symmetric monoidal category of finite sets, where the monoidal structure is disjoint union.  Let $K$ be any set, the over category $\fin_{/K}$ has objects given by pairs $( I, \col_I)$ where $I$ is a finite set and $\col_I \colon I \to K$ is any function. A map $f \colon (I,\col_I) \to (J,\col_J)$ is a map of sets $f \colon I \to J$ such that $\col_J f = \col_I$.  We call $K$ the \emph{set of colors}, and we call $\col_I$ the \emph{coloring} of $I$.  We usually leave the coloring implicit and simply refer to $I$ as a \emph{$K$-colored set}.  A map of $K$-colored sets is called a \emph{color-preserving} map or \emph{colored} map.  If we do not require $f$ to preserve the coloring, we call it an \emph{uncolored} map.  Disjoint union of sets over $K$ defines on $\fin_{/K}$ the structure of a symmetric monoidal category.
\end{definition}

The following definition is an amalgamation of those found in \cite{kontsevich-soibelman-deformations}, \cite{costello-a-infinity} and \cite{getzler-operads-revisited}.

\begin{definition}\label{definition:category-of-forests}
A \emph{$K$-colored young forest} is an uncolored map of finite $K$-colored  sets $x \colon I_x \to J_x$.  A \emph{$K$-colored forest} $f \colon x \to y$ is a color-preserving isomorphism $f\colon I_y\sqcup J_x \to J_y \sqcup I_x$ such that $y$ is realized by a colimit of $f$ and $x$ as in diagram~\ref{diagram:forest-colimit}.
\begin{equation}\label{diagram:forest-colimit}
  \begin{tikzpicture}[numbered picture, every node/.style={inner sep=2pt},every left delimiter/.style={xshift=1.5ex}, every right delimiter/.style={xshift=-1ex}]
      \matrix (m) [matrix of math nodes, row sep=.7ex, column sep=2em, text height=1.5ex, text depth=0.05ex, left delimiter=\lbrack, right delimiter=\rbrack]
      {
        & I_x
      \\
         J_y \sqcup I_x &
      \\
        & J_x
      \\};
      \foreach \source/\target in {1-2/2-1, 3-2/2-1, 1-2/3-2}
      {\path[->] (m-\source) edge (m-\target);}
      \foreach \source/\target/\pos/\maplabel in {3-2/2-1/below/f, 1-2/3-2/right/x}
      {\path (m-\source) -- node[map name, \pos] (\source/\target) {$\maplabel$} (m-\target);}
      \path (m.west) --++ (-12pt,1pt) node (colim) {$\colim$};
      \path (m-1-2) --++ (-4cm,0) node (Iy) {$I_y$};
      \path (m-3-2) --+ (-4cm, 0) node (Jy) {$J_y$};
      \path[->] (Iy) edge node[above, map name] {$f$} (colim) edge node[left,map name] {$y$} (Jy) (Jy) edge node[above, map name] {$\sim$} (colim);
  \end{tikzpicture}
\end{equation}
More specifically we require the map from $J_y$ to be an isomorphism and the corresponding composite with the inverse to be $y \colon I_y \to J_y$.  Now if $g \colon y \to z$ is another forest, then $gf \colon I_z \sqcup J_x \to J_z \sqcup I_x$ is defined by diagram~\ref{diagram:forest-composition}
\begin{equation}\label{diagram:forest-composition}
  \begin{tikzpicture}[numbered picture, every node/.style={inner sep=2pt}]
      \matrix (m) [matrix of math nodes, row sep=.7ex, column sep=2em, text height=1ex, text depth=0.05ex]
      {
        I_z &               & I_y &               & I_x
      \\
            & J_z\sqcup I_y &     & J_y\sqcup I_x &
      \\
        J_z &               & J_y &               & J_x
      \\};
      \foreach \source/\target in {1-1/2-2, 3-1/2-2, 1-3/2-2, 3-3/2-2, 1-3/2-4, 3-3/2-4, 1-5/2-4, 3-5/2-4}
      {\path[->] (m-\source) edge (m-\target);}
      \foreach \source/\target/\pos/\maplabel in {1-1/2-2/above/g, 3-3/2-2/above/g, 1-3/2-4/above/f, 3-5/2-4/above/f}
      {\path (m-\source) -- node[map name, \pos] (\source/\target) {$\maplabel$} (m-\target);}
  \end{tikzpicture}
\end{equation}
By this we mean take the colimit of the square in the diagram.  The inclusions from $J_z$ and $I_x$ will induce an isomorphism with this colimit.  Indeed, if $i \in I_y$ and $f(i) \in J_y$, then diagram~\ref{diagram:forest-colimit} implies that $f(i) = y(i)$.  Moreover, the condition on $g$ and $y$ implies that a sufficiently large power of the endomorphism $(\id_{J_z}, gy)$ of $J_z \sqcup I_y$ has image contained in $J_z$.  The condition on $gf$ can be checked by inserting the map $x$ into the right hand side of diagram~\ref{diagram:forest-composition}, then taking colimits.  We denote the natural maps given by the colimit in \ref{diagram:forest-colimit} as $[f|x]\colon I_y \sqcup J_y \sqcup I_x \sqcup J_x\to J_y$.

%All unlabeled maps above are the evident inclusions. Extend $N$ to an endomorphism $N \colon J_y \sqcup J_x \to J_y \sqcup J_x$ using the identity on $J_y$ and $N$ on $J_x$.  We require the limit endomorphism $N^\infty$ to exist and to factor through $J_y$.  This requirement defines the map $N^\infty$ in diagram~\ref{diagram:forest}.  Furthermore, we require both triangles to commute and the square to be a pullback.

%If $g \colon x \to y$ and $f \colon y \to z$ are forests, the composite forest $fg \colon x \to z$ is defined using the following construction.  Suppose we have maps of finite sets $p \colon B \to C\sqcup D$ and $q \colon C \to A\sqcup B$.  Define $(p,q)$ to be the endomorphism of $A\sqcup B \sqcup C \sqcup D$ given by $(\id_A, p, q, \id_D)$.  If the limit of iterates, $\lim_k (p,q)^k$, exists it must give a map $(p,q)^\infty\colon A\sqcup B \sqcup C \sqcup D \to A \sqcup D$.  With this notation set $\iota_{fg} = (\iota_f, s_g)^\infty \circ \iota_g$, $s_{fg} = (\iota_f, s_g)^\infty \circ s_f$, and $N_{fg} = (N_g, N_f \cdot \iota_f\inv)^\infty \circ (N_g \sqcup N_f \cdot \iota_f\inv)$.  In the last equation, the map $N_f \cdot \iota_f\inv \colon I_z \sqcup J_y \to J_z \sqcup I_y$ is defined by setting $(N_f\cdot \iota_f\inv)(v)$ to be $N_f(v)$ if $N_f(v) \in J_z$ or to be $\iota_f\inv(v)$ if $N_f(v) \in J_y$.

  This composition law defines a category called $\rfor_K$ whose objects are young forests and whose morphisms are forests.  Under disjoint union of sets $\rfor_K$ becomes a symmetric monoidal category.

  Given a $K$-colored forest $f \colon x \to y$ we call $V(f)\coloneqq J_x$ the set of \emph{internal vertices} of $f$.  We call $\inverts(f)\coloneqq I_y$ the set of \emph{input vertices} of $f$ and $\rt(f)\coloneqq J_y$ the set of \emph{root vertices} of $f$.

  Given a forest $f \colon x \to y$ we define $\Exedge(f)\coloneqq J_y \sqcup I_x$ to be the set of \emph{extended edges} of $f$.  The isomorphism $f \colon I_y \sqcup J_x \to \Exedge(f)$ gives a decomposition
  \[
  	\Exedge(f) = (I_y \times_f J_y) \sqcup (I_y \times_f I_x) \sqcup (J_x \times_f J_y) \sqcup (J_x \times_f I_x).
  \]
  We call $\unit(f)\coloneqq I_y \times_f J_y$ the set of \emph{trivial} or \emph{unit} edges; $\leaf(f) \coloneqq I_y \times_f I_x$ the set of \emph{leaf} edges or the set of \emph{leaves}; $\rt(f)\coloneqq J_x \times_f J_y$ the set of \emph{root} or \emph{output} edges; and $\Edge(f) \coloneqq J_x \times_f I_x$ the set of \emph{internal edges} or simply \emph{edges}.  The internal edges have the most compact notation because they will be referred to the most often.  The inclusion of edges into extended edges will be denoted $\edge(f)\colon \Edge(f) \hookrightarrow \Exedge(f)$.  The extended edges can be remembered by taking four pullbacks in the standard diagram representation of $f$ as in diagram~\ref{diagram:extended-edges}.
	\begin{equation}\label{diagram:extended-edges}
    \begin{tikzpicture}[->,numbered picture, node distance=4mm and 15mm, every node/.style={on grid,inner sep=0.5pt}, lbl/.style={above=1pt,font=\scriptsize}]
        \node (JyIx) {$\Exedge(f)$};
        \node[above left =of JyIx] (Iy) {$I_y$} edge[->,pos=0.75] node[lbl] {$f$} (JyIx);
        \node[above right=of JyIx] (Ix) {$I_x$} edge[->] (JyIx);
        \node[below right=of JyIx] (Jx) {$J_x$} edge[->] node[lbl] {$f$} (JyIx);
        \node[below left =of JyIx] (Jy) {$J_y$} edge[->] (JyIx);
        \node[above right=of Iy]   {$\leaf(f)$} edge (Iy) edge (Ix);
        \node[below left =of Iy]   {$\unit(f)$} edge (Iy) edge (Jy);
        \node[below left =of Jx]   {$\rt(f)$}   edge (Jy) edge (Jx);
        \node[above right=of Jx]   {$\Edge(f)$}     edge (Ix) edge (Jx);
    \end{tikzpicture}
	\end{equation}
	
\end{definition}

\begin{remark}
  \label{remark:drop-K-from-notation}
Throughout this paper $K$ denotes a finite set of colors.  We often drop it from the notation.  Forests and young forests are always $K$-colored and $\for_K$ will be abbreviated $\for$.
\end{remark}

%\begin{definition}
%\label{definition:morphism-of-forests}
%A morphism of forests $F \to F'$ is a map $V_{\rm ext}(F) \to V_{\rm ext}(F')$ preserving the decompositions $V_{\rm ext} = V_r \sqcup V \sqcup V_t$, commuting with $N$, and preserving the coloring.
%
%An isomorphism is a bijection on extended vertices.  The group of isomorphisms from $F$ to itself is denoted $\aut(F)$.  There is an evident homomorphism $\aut(F) \to \aut(V_t(F))$, where automorphisms of $V_t(F)$ preserve the coloring.
%
%Given a $K$-colored set $I \rightarrow K$ and $k \in K$, the isomorphism class of trees $T$ with $\mathrm{in}(T) = I$ and $\mathrm{out}(T)= \{k\}$ as $K$-colored sets is denoted by $\trees(I;k)$.
%\end{definition}
%
%Given a tree $T_1$, an incoming edge $\epsilon=(v, N(v))\in \mathrm{in}(T_1)$ and a tree $T_2$ with $\mathrm{col}(\mathrm{out}(T_2))=\mathrm{col}(\epsilon)$ we can define $T_1 \circ_\epsilon T_2$ by gluing the root vertex $v_r$ of $T_2$ to $N(v)$ and gluing $N^{-1}(v_r)$ to $v$.  This is called \emph{grafting} the tree $T_2$ to the tree $T_1$ at the edge $\epsilon$.

\section{The \texorpdfstring{$W$}{W} construction}
\label{section:w-construction}
Let $(\topl,\times)$ be the symmetric monoidal category of compactly generated topological spaces with the Cartesian product.  We show how the $\W$ construction of Boardman and Vogt \cite{boardman-vogt-hiasots} can be realized as a coend construction using $\for$.  In nice situations the $W$ construction gives a cofibrant replacement for an operad, as shown in \cite{berger-moerdijk-w}.  We use $[0,\infty]$ as our edge labels as in \cite{kontsevich-operads-motives}.

Given any young forest $z$ there is a contravariant functor $W \colon \rfor_{/z}\to \topl$ from the over category of $z$ to topological spaces.  For any object $g \colon y \to z$ of this over category, set $W(g) = \map(E(g), [0,\infty])$.  If $f \colon x \to y$ is a forest, there is a co-correspondence
\[
	E(g) \hookrightarrow I_y \xrightarrow{[g|f]} J_z \sqcup I_x \hookleftarrow I_x \hookleftarrow E(gf),
\]
where the map $[g|f]\colon I_z \sqcup J_z \sqcup I_y \sqcup J_y \sqcup I_x \sqcup J_x \to J_z \sqcup I_x$ is given by including into the square in diagram~\ref{diagram:forest-composition} then passing to the colimit.  By pushing forward and pulling back functions, we get a map $W_\Sigma(f) \colon W(g) \to W(gf)$.  This process uses the sum operation on $[0,\infty]$.  This is an extension of $+$ on $[0,\infty)$ such that $t + \infty = \infty = \infty + t$ for all values of $t$.  Concretely, any morphism $gf \to g$ in the over category collapses some internal edges and deletes some unary vertices.  Any collapsed edge in $gf$ is labeled with 0 in $W(gf)$.  Any vertex of $gf$ which is deleted in $g$ causes two edges of $gf$ to have the same image under $[g|f]$.  This merged edges in $gf$ is labeled with the sum of the two edges in its pre image under $[g|f]$.

If $h \colon z \to w$ is a forest there is a morphism $W_\infty(h) \colon W(g) \to W(hg)$ which uses the maps $E(g) \hookrightarrow I_y \hookleftarrow E(hg)$.  In this case we do not push forward and pull back functions.  Rather we extend a function $t \colon E(g) \to [0,\infty]$ to a function on $E(hg)$ by setting $t(\epsilon) = \infty$ if $\epsilon \not\in E(g)$.  This defines a natural transformation $W_\infty(h) \colon W \to W h_*$ where $h_* \colon \rfor^{op}_{/z} \to \rfor^{op}_{/w}$ is induced by $h \colon z \to w$.  In other words, diagram~\ref{diagram:Wg-Wgf-Whgf} commutes for all composeable triples $h, g$, and $f$.
\begin{equation}\label{diagram:Wg-Wgf-Whgf}
       \begin{tikzpicture}[numbered picture]
          \matrix (m) [matrix of math nodes, row sep={2.5ex}, column sep={2.5em},
             text height=1.5ex, text depth=0.25ex]
          {
           W(g) & W(gf)
          \\
           W(hg) & W(hgf)
          \\};
          \path[->, every node/.style={map name}]
           (m-1-1) edge node[above] {$W_\Sigma (f)$} (m-1-2) edge node[left] {$W_\infty(h)$} (m-2-1)
           (m-1-2) edge node[right] {$W_\infty (h)$} (m-2-2)
           (m-2-1) edge node[above] {$W_\Sigma (f)$} (m-2-2);
       \end{tikzpicture}
    \end{equation}

\begin{definition}
  \label{definition:an-operad-is-a-functor}
  A $K$-colored operad $\calO$ is a symmetric monoidal functor $(\rfor, \sqcup) \to (\topl,\times)$.
\end{definition}

Consider an operad $\calO$ as a collection of functors $\calO_z \colon \rfor_{/z} \to \topl$ by setting $\calO_z(g) \coloneqq \calO(y)$ for $g\colon y \to z$.  For a young forest $z$ the topological space $W\calO(z)$ is the coend
\begin{equation}\label{equation:WO(z)-as-coend}
\W \calO (z) = W\otimes_{\rfor_{/z}}\calO_z = \left(\coprod_{g \colon y \to z} W(g) \times \calO(y) \right)\modsim
\end{equation}
where $(W_{\Sigma}(f)t, \alpha)\sim(t, \calO(f)\alpha)$ for every $f \colon x \to y$, $t \in W(g)$ and $\alpha \in \calO(x)$.
Now a forest $h \colon z \to w$ with the natural transformations above gives us a map
\begin{equation}
  \label{equation:WO(z)-to-WO(w)}
  W\calO(z) = W \!\otimes_{\rfor_{/z}} \!\calO_z \xrightarrow{W_\infty(h) \otimes \id} W h_* \!\otimes_{\rfor_{/z}}\!\calO_z \to W\!\otimes_{\rfor_{/w}} \!\calO_w = W\calO(w).
\end{equation}
This defines $W\calO$ as a functor $\rfor \to \topl$.  This functor is symmetric monoidal, so $\W\calO$ is a $K$-colored operad.

In the sequel, we will define several variants on the category $\for$.  Each of these variants admits a functor to $\for$ and we want a corresponding $W$ construction for each.
\begin{definition}
	\label{definition:WvarO}
	If $\for^{\mathrm{var}}$ is any symmetric monoidal category equipped with a functor $\for^{\mathrm{var}} \to \for$, let $W_{\mathrm{var}}\calO \colon \rfor^{\mathrm{var}} \to \topl$ be defined by \ref{equation:WO(z)-as-coend} and \ref{equation:WO(z)-to-WO(w)} where use the given functor to replace $\for_{/z}$ by $\for^{\mathrm{var}}_{/v}$ for an object $v \in \for^{\mathrm{var}}$.
\end{definition}

\subsection{\texorpdfstring{$\W\calO$}{WO} is cofibrant}
\label{section:WO-is-cofibrant}

\begin{definition}
  \label{definition:weighted-forests}
  Let $f \colon x \to y$ be a forest.  For $I \subset I_y \sqcup I_x \sqcup J_x$ and $j \in J_y$ let $I(j)$ denote $I \cap [f|x]\inv(j)$, the set of elements of $I$ living over $j$.   A \emph{weighted young forest} is a pair $(x, \omega_x)$ where $x$ is a young forest and $\omega_x \colon J_x \to \mathbb Z_{\geq 0}$ is any function, called the \emph{weight of $x$}.    A \emph{weighted forest} $f \colon (x, \omega_x) \to (y, \omega_y)$ is a forest $f \colon x \to y$ such that for all $j \in J_y$,
  \begin{equation}
    \label{equation:omega-y-in-terms-of-f-and-omega-x}
    \omega_y(j) \geq \#E(f) + \sum_{{i \in J_x(j)}} \!\!\!\omega_x(i).
  \end{equation}
\end{definition}

%To verify that the above formula is well-defined, consider that $I_x \cong \Edge(f) \sqcup \leaf(f)$, $J_x = \Edge(f) \sqcup \rt(f)$, and $I_y = \leaf(f)\sqcup \unit(f)$.  We can rewrite equation~\ref{equation:omega-y-in-terms-of-f-and-omega-x} as
%\begin{equation}
%  \label{equation:omega-y-formula-using-edges}
%  \omega_y(j) =  \sum_{i \in \rmE(f)(j)} \!\!(1+\omega_x(i)) + \sum_{i \in \rt(f)(j)} \!\!\omega_x(i) \,\,-\, \rmun(f)(j).
%\end{equation}
%Since $\rt(f)(j) \sqcup \rmun(f)(j) \simeq \{j\}$, and $1+\omega_x(i) \geq 0$ the right hand side of \ref{equation:omega-y-formula-using-edges} is at least $-1$.  If $j \not\in J^1_y$, then either $y\inv(j) = \emptyset$ or $y \inv(j)$ has size at least two.  If $y\inv(j) = \emptyset$, then $I_y(j) = \emptyset$.  Since $x\inv \colon J_x^1(j) \to I_x(j)$ is injective, equation~\ref{equation:omega-y-in-terms-of-f-and-omega-x} shows that $\omega_y(j) \geq 0$ in this case.  If $y \inv(j)$ has size at least two then $\rmun(f)(j)=\emptyset$.  and  equation~\ref{equation:omega-y-formula-using-edges} becomes $-1+\sum_{i\in J_x(j)} (1+\omega_x(i))$.  Also, $\#y\inv(j) \geq 2$ implies $\#x\inv(i)\geq 2$ for some $i \in J_x(j)$, so $\omega_y(i)\geq 0$ in this case as well.

If $g \colon (y,\omega_y) \to (z, \omega_z)$ is a weighted forest.  Then one can show that for each $j \in J_z$, we have
\begin{align}
\label{equation:E(gf)-is-E(g)-plus-E(f)-minus-un(f)}
	\#\rmE(gf)(j) &= \#\rmE(g)(j) + \sum_{i \in J_y(j)} \#\rmE(f)(i)-\#\rmun(f)(i) \\
\notag
 & = \#\rmE(g)(j) + \#\rmE(f)(j)  - \#\rmun(f)(j).
\end{align}
This implies $gf \colon x \to z$ defines a weighted forest $gf \colon (x, \omega_x) \to (z, \omega_z)$.

\begin{remark}
	\label{remark:weight-limits-internal-edges}
	If $f \colon (x, \omega_x) \to (y,\omega_y)$ is a weighted forest and $\omega_y \leq k$ then certainly $\#E(f) \leq k$.  Furthermore, if $g \colon (y, \omega_y) \to (z, \omega_z)$ is a weighted forest where $\omega_z \leq k+1$ and $\#E(gf)(j) = k+1$ for some $j \in J_z$, then equations~\ref{equation:omega-y-in-terms-of-f-and-omega-x} and \ref{equation:E(gf)-is-E(g)-plus-E(f)-minus-un(f)} imply that $\omega_x(i) = 0$ for all $i \in J_x(j)$, $\#\rmun(f)(j) = 0$, and $\#E(g) \geq 1$.
\end{remark}

\begin{definition}
  Disjoint union of forests extends to disjoint union of weighted forests. Let $\rfor_\omega$ denote the symmetric monoidal category of weighted forests.
  For each $k \geq 0$, let $\rfor_{k}$ denote the full subcategory of $\rfor_\omega$ generated by objects of the form $(x,\omega_x)$ such that $\omega_x(j) \leq k$ for every $j \in J_x$.  Let $\Op{k}$ denote the category of \emph{weight $k$ operads}, which are symmetric monoidal functors $\rfor_k \to \topl$.   The category $\Op{0}$ is often denoted $\Coll$ in the literature and is called the category of \emph{$K$-colored pointed collections} in $\topl$.  All of these categories are endowed with the projective model structure via the forgetful functor $\for_k \to \Coll$.  In particular, left adjoints preserve cofibrations.
\end{definition}

We will often refer to the weighted forest $(x, \omega_x)$ simply as  $x$ when the weight is understood. If $\omega_x(j) \leq k$ for all $j \in J_x$, we write $\omega_x \leq k$ or say $x$ \emph{has weight } $\leq k$.  Note that if $f \colon x \to y$ is a weighted forest, then $x$ has weight $\leq k$ if $y$ has weight $\leq k$.

\begin{remark}
  \label{remark:expound for0}
  The only morphisms in the category $\rfor_0$ are those $f \colon x \to y$ which satisfy $f(J_x) \subset J_y$.  In other words, $f$ has no internal edges.  However $f$ may still have some unit edges.  Every young forest $x$ is a disjoint union of young forests of the form $I \to \{c\}$, i.e. young \emph{trees}.  Denote this young tree by $(I;c)$.  A pointed collection $\calO \colon \rfor_{0} \to \topl$ consists of spaces $\calO(I;c)$ for every young tree $(I;c)$.  In addition the trees $f \colon (I;c) \to (I;c)$ with no unit edges define an action of $\aut(I)$ on $\calO(I;c)$.  Each unit tree $(\emptyset; \emptyset) \to (c;c)$ defines a map $\ast \to \calO(c;c)$ which, if $\calO$ is an operad, is the $c$-colored unit of $\calO$.
\end{remark}

The categories $\rfor_k$ define a functor from the poset $\mathbb Z_{\geq 0}$ to the category of symmetric monoidal categories.  That is, for each $k \leq \ell$ there are inclusions $\rfor_{k} \to \rfor_{\ell}$, and these inclusions are all compatible with one another.   The colimit of this these categories is $\rfor_\omega$.  There are forgetful functors $\rfor_{k} \to \rfor$, compatible with the inclusions $\rfor_{k} \to \rfor_{\ell}$, and the induced map from the colimit is the canonical forgetful functor $\rfor_{\omega} \to \rfor$.  For $k \leq \ell$ let $U^{\ell}_k$ denote the induced forgetful functor $\Op{\ell} \to \Op{k}$, and let $U_k \colon \Op \to \Op{k}$ be the functor sending an operad $\calO$ to \emph{underlying pointed collection of $\calO$}.  Let $F^k_\ell$ and $F^k$ denote the left adjoints of $U_k^\ell$ and $U_k$ respectively.

%If $z$ is a young forest of weight $\leq k$ let $\for_{k/z}$ denote the over category of $(z, \omega_z)\in \for_k$.  The forgetful functor $\for_{k/z} \to \for_{/z}$ allows us to define  $W \colon \for_{k/z}^{op} \to \topl$ by pulling back $W \colon \for_{/z}^{op} \to \topl$.  Also for an operad $\calO \colon \rfor \to \topl$ define $\calO_z \colon \for_{k/z} \to \topl$ by pulling back $\calO_z \colon \for_{/z} \to \topl$.

Let $\calO$ be an operad.  Since each $k\geq 0$ defines a symmetric monoidal category $\for_k$ equipped with a symmetric monoidal functor $\for_k \to \for$ definition~\ref{definition:WvarO} to define $W_k\calO\colon\rfor_{k} \to \topl$.  Precisely, this is the weight $k$ operad which sends $(z, \omega_z)$ to
\begin{equation}
  \label{equation:WkO(z)}
  \W_k\calO(z,\omega_z) \coloneqq \W \otimes_{\rfor_{k/z} } \calO_z = \Bigg(\raise1.5ex\hbox{$\displaystyle\coprod_{g \colon (y,\omega_y) \to (z, \omega_z)} \W(g) \times \calO(y) $}\Bigg)\modsim,
\end{equation}
where $(W_\Sigma(f)t, \alpha) \sim (t, \calO(f)\alpha)$ for $f \colon (x, \omega_x) \to (y, \omega_y)$, $t \in W(g)$, and $\alpha \in \calO(x)$.  Note that $x$ and $y$ necessarily have weight $\leq k$ since $z$ does.  If $w$ is a young forest of weight $\leq k$ and $h \colon (z, \omega_z) \to (w, \omega_w)$ is a weighted forest, then $W_k\calO(h) \colon W_k\calO(z) \to W_k\calO(w)$ is defined exactly as in \ref{equation:WO(z)-to-WO(w)}.

The inclusion $\rfor_{k} \to \rfor_{k+1}$ induces a map of operads $W_{k} \calO \to U^{k+1}_{k} W_{k+1}\calO$.  To describe the adjoint of this map, note that
\begin{equation}
  \label{equation:FkWO(z)}
  F^{k}_{k+1} W_{k}\calO(z,\omega_z) = \Bigg(\raise2ex\hbox{$\displaystyle\coprod_{\substack{f \colon (x, \omega_x) \to (y, \omega_y)\\ g \colon (y, \omega_y) \to (z,\omega_z)}} \!\!\!\!\! W(f) \times \calO(x)$}\Bigg)\modsim,
\end{equation}
  where $\omega_x, \omega_y \leq k$ and $\omega_z \leq k+1$.  %If $q \colon x' \to x$, $f \colon x \to y'$, $p \colon y' \to y$, $g \colon y \to z$, and $(t, \alpha) \in W(f) \times \calO(x)$ then $((gp,f), t, \alpha) \sim ((g,pf), W_\infty(p)t, \alpha)$.  Also, if $(t, \alpha') \in W(f) \times \calO(x')$ then $((gp,f), t, \calO(q)\alpha') \sim ((gp,fq), W_{\Sigma}(q)t, \alpha')$.  The map $F^{k-1}_kW_{k-1}\calO(z, \omega_z) \to W_k(z, \omega_z)$ sends $((g,f), t, \alpha)$ to $(gf, W_\infty(g)t, \alpha)$.  Clearly if $\omega_z \leq k-1$ then $F^{k-1}_k W_{k-1}\calO(z, \omega_z) \to W_k\calO(z, \omega_z)$ is an isomorphism.
To describe the relation $\sim$ we describe maps $\mu \colon F^k_{k+1}W_k \calO(z) \to X$ for an arbitrary space $X$.  Such a map is given by a collection $\{\mu(g,f) \colon W(f) \times \calO(y) \to X\}$ where $g \colon y \to z$ and $f \colon x \to y$ are weighted forests, and $\omega_x, \omega_y \leq k$.  These maps must make the diagrams in \ref{diagram:FWO-to-X-two-squares} commute for all $q \colon x' \to x$, $p \colon y' \to y$, and $f' \colon x' \to y'$ in $\for_k$.
\begin{equation}
  \label{diagram:FWO-to-X-two-squares}
  \begin{tikzpicture}[numbered picture, node distance=6ex and 10em,
  math nodes, text height=2ex, text depth=0.75ex,
  every node/.style={on grid,inner sep=1pt},
  mn/.style={inner sep=1pt, font=\scriptsize}]
    \node (ul) {W(f)\!\times\! \calO(x')};
    \node[right=of ul] (ur) {W(fq)\!\times\! \calO(x')};
    \node[below=of ul] (ll) {W(f)\!\times\!\calO(x)};
    \node[right=of ll] (lr) {X};
    \path[->]
      (ul) edge node[above,mn] {W_\Sigma(q)\!\times\! 1} (ur)
      (ul) edge node[left,mn] {1 \!\times\! \calO(q)} (ll)
      (ur) edge node[right,mn,inner sep=3pt] {\mu(g,fq)} (lr)
      (ll) edge node[above,mn] {\mu(f,g)} (lr);
    \begin{scope}[node distance=6ex and 4.5em]
    \node[right=2.4cm of ur] (ul) {W(f')\!\times\!\calO(x')};
    \node[below right=of ul] (lm) {X};
    \node[above right=of lm] (ur) {W(pf')\!\times\!\calO(x')};
    \end{scope}
    \path[->]
      (ul) edge node[above,mn] {W_\infty(p)} (ur)
      (ul) edge node[left,mn,inner sep=3.5pt] {\mu(gp,f')} (lm)
      (ur) edge node[right,mn,inner sep=1.5pt] {\mu(g,pf')} (lm);
  \end{tikzpicture}
\end{equation}
Given $g$ and $f$ as above, let $\iota(g,f) \colon W(f) \times \calO(x) \to W(gf) \times \calO(x)$ be $W_\infty(g) \times 1$.  If $((g,f), t, \alpha)$ represents a point $\beta$ of $F^k_{k+1}W_k\calO(z,\omega_z)$ then its image in $W_{k+1}\calO(z, \omega_z)$ is represented by $\iota(g,f)((g,f), t, \alpha) = (gf, W_\infty(g)t, \alpha)$.

A map $W_{k+1}\calO(z) \to X$, where $\omega_z \leq k+1$, consists of a collection maps $\{\eta(g) \colon W(g) \times \calO(y) \to X\}$ indexed by the set of all weighted forests $g \colon y \to z$ in $\for_{k+1}$.  These maps must make diagram~\ref{diagram:WO-to-X-square} commute for every weighted forest $q \colon y' \to y$ in $\for_{k+1}$.
\begin{equation}
  \label{diagram:WO-to-X-square}
  \begin{tikzpicture}[numbered picture, node distance=6ex and 12em,
  math nodes, text height=1.5ex, text depth=0.75ex,
  every node/.style={on grid,inner sep=1pt},
  mn/.style={inner sep=1pt, font=\scriptsize}]
    \node (ul) {W(g) \times \calO(y')};
    \node[right=of ul] (ur) {W(gq) \times \calO(y')};
    \node[below=of ul] (ll) {W(g) \times \calO(y)};
    \node[right=of ll] (lr) {X};
    \path[->]
      (ul) edge node[above,mn] {W_\Sigma(q) \times 1}(ur)
      (ul) edge node[left,mn] {1 \times \calO(q)} (ll)
      (ur) edge node[right,mn] {\eta(gq)} (lr)
      (ll) edge node[above,mn] {\eta(g)} (lr);
  \end{tikzpicture}
\end{equation}

%Let $(/h)$ denote the over category $(\for_{k+1}/z)_{/h}$. Clearly $(/h)$ is isomorphic to $\for_{k+1/x}$.  So we can define a functor $\calO_h\colon (/h) \to \topl$ by sending a weighted forest $p \colon u \to x$ to $\calO(u)$.  If $p_i \colon u_i \to x$, $i=1,2$ and $p^1_2 \colon u_2 \to u_1$ satisfy $p_2 = p_1 p^1_2$ then $p^1_2$ defines a $(/h)$-morphism $p_2 \to p_1$.  Define $\calO_h(p_2^1)$ to be $\calO(p_2^1) \colon \calO(u_2) \to \calO(u_1)$ defines a morphism $\calO_h(pq) \to \calO_h(p)$.

%Let $h \colon x \to z$ be a forest in $\for_{k+1}$, thought of as an object of $\for_{k+1/z}$.  Let $(h/)$ denote the under category $(\for_{k+1/z})_{h/}$. This is the category of factorizations of $h$.  Indeed, if $gf = h$, for some forests $g \colon y \to z$ and  $f \colon x \to y$ in $\for_{k+1}$, then $g$ is an object of $\for_{k+1/z}$ and $f$ defines an object $f \colon h \to g$ of $(h/)$.  A morphism $f^1 \to f^2$ where $f^i \colon h \to g_i$ and $g_i \colon y_i \to z$ is a weighted forest $q_2^1 \colon y_2 \to y_1$ such that $f^1= q_2^1f^2$ and $g_1q_2^1 = g_2$.  Define $W_h \colon (h/) \to \topl$ by sending $f^i$ to $W(f^i)$ and sending $q_2^1$ to $W_\infty(q_2^1) \colon W(f^2) \to W(q_2^1f^2)$.  Finally, define $W^h \colon (h/)^{op} \to \topl$ by $W^h(f^i) = W(g_i)$ and $W^h(q_2^1) = W_\Sigma(q_2^1) \colon W(g_1) \to W(g_1q_2^1)$.

%Let $\abs{\cdot} \colon (/h) \to \mathbb Z_{\geq 0}$ denote the number of unit edges $\#\rmun(p)$ of a weighted forest $p \colon u \to x$, thought of as an object $p \colon hp \to h$ in $(/h)$.  Note that if

For $g \colon y \to z$ a forest in $\for_{k+1}$, let $(W\times \calO)^+_k(g)$ be $W(g) \times \calO(y)$ if $g$ has $\leq k$ internal edges.  Otherwise let $(W\times \calO)^+_{k}(g) \subset W(g) \times \calO(y)$ be the set of $(t,\alpha)$ such that $t(\epsilon) = 0$ or $t(\epsilon)=\infty$ or $\alpha(j) = \id$ for some $\epsilon \in E(g)$ or $j \in V(g)$.  Define the map $(W\times \calO)^+_{k}(g) \to F^k_{k+1}W_k\calO(z)$ by collapsing any edge labeled $0$ and deleting any vertex labeled with the identity.  The square in diagram~\ref{diagram:FWO-to-WO-pushout} is a pushout.  Indeed, if we are given $\eta_T([g]) \colon (W(g) \times \calO(\dom g))_{\aut(g)} \to T$ for every $[g]\in \pi_0 \for_{k+1/z}$ and $\mu_T \colon F^k_{k+1}W_k\calO(z) \to T$, such that
\begin{equation}
	\label{diagram:FWO-to-WO-pushout}
\begin{tikzpicture}[numbered picture, node distance=8ex and 15em, display math nodes, text height=2ex, text depth=0.75ex,every node/.style={on grid,inner sep=1pt},mn/.style={inner sep=1pt, font=\scriptsize}]
	\node (ul) {\coprod_{[g] \in \pi_0 \for_{k+1/z}} ((W\times\calO)^+_{k}(g))_{\aut(g)} };
\node[right=of ul] (ur) {F^k_{k+1}W_k\calO(z)};
\node[below=of ul] (ll) {\coprod_{[g]\in\pi_0\for_{k+1/z}} (W(g) \times \calO(\dom g))_{\aut(g)}};
\node[right=of ll] (lr) {W_{k+1}\calO(z)};
\path[->]
	(ul) edge (ur)
(ul) edge (ll)
(ur) edge (lr)
(ll) edge (lr);
\end{tikzpicture}
\medskip
\end{equation}
Using the techniques of Berger and Moerdijk \cite[section 2]{} one can also show that if $U_0\calO$ is a cofibrant pointed collection then $(W\times \calO)^+_k(g) \to W(g) \times \calO(\dom g)$ is an $\aut(g)$-cofibration, where $\aut(g)$ is the automorphism group of $g$ as an object of the category $\for_{k+1/z}$.  This implies the map on the left in \ref{diagram:FWO-to-WO-pushout} is an $\aut(z)$-cofibration.  Thus $U_0F^k_{k+1}W_k\calO \to U_0W_{k+1}\calO$ is a cofibration of pointed collections.

\begin{lemma}
	\label{lemma:FWO-to-WO-cofibration-of-weight-k-operads}
  The natural map $\iota_k\colon F^{k}_{k+1} W_{k}\calO \to W_{k+1}\calO$ is a cofibration of weight $k+1$ operads.
\end{lemma}
\begin{proof}	
	Consider a commutative diagram of weight $k$ operads where $\pi(z) \colon \calP(z) \to \calQ(z)$ is a fibration for every $z \in \for_{k+1}$,
  \begin{equation}
    \label{diagram:commutative-diagram-of-weight-k-operads}
    \begin{tikzpicture}[numbered picture, node distance=6ex and 5em,
    math nodes, text height=2ex, text depth=0.75ex,
    every node/.style={on grid,inner sep=1pt},
    mn/.style={inner sep=1pt, font=\scriptsize}]
      \node (ul) {F^{k}_{k+1}W_{k}\calO};
      \node[right=of ul] (ur) {\calP};
      \node[below=of ul] (ll) {W_{k+1}\calO};
      \node[right=of ll] (lr) {\calQ.};
      \path[->]
        (ul) edge node[above,mn] {\eta} (ur)
        (ul) edge node[left,mn] {\iota} (ll)
        (ur) edge node[right,mn] {\pi} (lr)
        (ll) edge node[above,mn] {\mu} (lr);
    \end{tikzpicture}
  \end{equation}
  By the discussion below diagram~\ref{diagram:FWO-to-WO-pushout}, there is a lift on the level of pointed collections, $\nu \colon U_0\W_{k+1}\calO \to U_0\calP$, $\nu\iota = \eta$, $\pi\nu = \mu$.  We will show that this is automatically a lift on the level of weight $k+1$ operads.  The condition that $\nu$ is a morphism of pointed collections means that the square on the right in diagram~\ref{diagram:lift-of-collections-is-lift-of-operads} commutes for all $h \colon z \to z'$ with no internal edges, where $z$, $z'$ are objects of $\for_{k+1}$.  We want to show that the square on the right in \ref{diagram:lift-of-collections-is-lift-of-operads} commutes for all $h \colon z \to z'$ in $\for_{k+1}$.  By assumption, the square on the left, and the top and bottom triangles commute.
  \begin{equation}
  	\label{diagram:lift-of-collections-is-lift-of-operads}
    \begin{tikzpicture}[numbered picture, node distance=7ex and 7em, math nodes, text height=2ex, text depth=0.75ex, every node/.style={on grid,inner sep=1pt},mn/.style={inner sep=1pt, font=\scriptsize}]
  		\node (ul) {W_{k+1}\calO(z)};
    	\node[right=of ul] (ur) {\calP(z)};
    	\node[below=of ul] (ll) {W_{k+1}\calO(z')};
    	\node[right=of ll] (lr) {\calP(z')};
    	\begin{scope}[node distance=7ex and 10em]
    		\node[left=of ul]  (ull) {F^{k}_{k+1}W_{k}\calO(z)};
    		\node[below=of ull] (lll) {F^{k}_{k+1}W_{k}\calO(z')};
    	\end{scope}
			\path[->]
				(ul) edge node[above, mn] {\nu(z)} (ur)
				(ul) edge node[left, mn] {W_{k+1}\calO(h)} (ll)
				(ur) edge node[right, mn] {\calP(h)} (lr)
				(ll) edge node[above, mn] {\nu(z')} (lr)
				(ull) edge node[above, mn] {\iota(z)} (ul)
				(ull) edge node[left, mn] {F^{k}_{k+1}W_{k}\calO(h)} (lll)
				(lll) edge node[below, mn] {\iota(z')} (ll)
				(ull) edge[bend left] node[above,mn] {\eta(z)} (ur)
				(lll) edge[bend right] node[above, mn] {\eta(z')} (lr);
		\end{tikzpicture}
	\end{equation}
	If $h$ has at least one internal edge, then $z$ has weight at most $k$.  This implies $\iota(z)$ is an isomorphism.  Thus,
  \begin{align*}
  		\nu(z')W_{k+1}\calO(h) &= \nu(z')\iota(z') F^{k}_{k+1}W_{k}\calO(h) (\iota(z)) \inv \\
  	& = \eta(z') F^{k}_{k+1}W_{k}\calO(h) (\iota(z))\inv \\
  	& = \calP(h) \eta(z)  (\iota(z))\inv \\
  	& = \calP(h) \nu(z).
  \end{align*}
\end{proof}

\begin{definition}
  \label{definition:WomegaO-and-WO}
  Let $W_\omega\calO$ denote the weight $\omega$ operad $\colim_k F^k_\omega W_k\calO$.  Let $W\calO = F^\omega W_\omega\calO \in \Op$. Let $\varepsilon \colon W\calO \to \calO$ denote the adjoint of the natural map $W_\omega\calO \to U_\omega \calO$.
\end{definition}

\begin{proposition}
  \label{proposition:WO-cofibrant}
  If $U_0\calO$ is a cofibrant collection, then the natural map $\epsilon \colon W\calO \to \calO$ is a weak equivalence and $F^0U_0 \calO \to W\calO$ is a cofibration.  In particular, $W\calO$ is a cofibrant operad.
\end{proposition}
\begin{proof}
	By induction, $F^0_kU_0 \calO \to W_k \calO$ is a cofibration in $\Op{k}$.  Indeed, $U_0 \calO = W_0\calO$ and if $F^0_{k-1}U_0 \calO \to W_{k-1} \calO$ is a cofibration in $\Op{k-1}$, then applying $F^{k-1}_k$ we get a cofibration in $\Op{k}$, $F^0_k U_0 \calO \to F^{k-1}_kW_{k-1}$.  Then lemma ?? shows that $F^{k-1}_k W_{k-1}\calO \to W_k\calO$ is a cofibration in $\Op{k}$.   Taking $\colim_k$, we get $F^0_\omega U_0 \calO \to W_\omega\calO$ a cofibration in $\Op{\omega}$.  Apply $F^\omega$, then $F^0U_0 \calO \to \W\calO$ is a cofibration in $\Op$.  Since $F^0$ preserves cofibrant objects, $F^0U_0 \calO$ is cofibrant.  Thus $\W\calO$ is cofibrant.  Finally, the map $U_0W\calO \to U_0\calO$ has homotopy inverse given by the adjoint of the inclusion of the free operad.	
\end{proof}

\section{The swiss cheese operad}
\label{section:the-swiss-cheese-operad}

Throughout the remainder of this paper, let $K=\{\f,\h\}$ be the set of colors.
Let $C_\h \hookrightarrow  \for$ be the full subcategory of $\for$ spanned by young forests $x$ with $J_x$ of color $\h$.  Let $C_\h^{\f1}$ denote the full subcategory of $C_\h$ spanned by young forests $x$ such that $\#_\f x\inv(j) \leq 1$ for all $j \in J_x$, where $\#_c I$ denotes the number of $c$-colored elements of a colored set $I$. Using the construction in definition~\ref{definition:WvarO} we can define the operad $W_\h \calO \colon C_\h \to \topl$ using the functor $C_\h \to \for$.  In the same way, define $W_{\h}^{\f1} \calO \colon C_{\h}^{\f1} \to \topl$ using the functor $C_{\h}^{\f1} \to \for$.

For $k \in \mathbb Z_{\geq -1} \sqcup \{ \omega\}$ define $C_k$ to be the full subcategory of weighted young forests $(x,\omega_x)$, where $x \in C_\h$, and $\omega_x(j) \leq k$ if $j\in J_x$ satisfies $\#_\f x\inv(j) \geq 2$.  For $f \colon x \to y$, we require condition~\ref{equation:omega-y-in-terms-of-f-and-omega-x}.  For each $k \geq -1$ define $W_{k}^C\calO \colon C_k \to \topl$ using the functor $C_k \to \for$ and the construction in definition~\ref{definition:WvarO}.  The functor categories $\Op^C_k, \Op_{\h}, \Op_{\h}^{\f1}$ are defined in the obvious way.  Adjunctions are denoted
\begin{align}
	\label{gather:adjunctions-notation}
	F^{\f1,k}_{\h,C} \colon \Op^C_k &\leftrightarrows  \Op^{\f1}_\h \colon U^{\h,C}_{\f1,k} &
	F^\ell_k \colon \Op^C_{\ell} &\leftrightarrows \Op^C_k \colon U^k_\ell \\
 	F^k_{\h, C} \colon \Op^{C}_{k} &\leftrightarrows \Op_\h \colon U_k^{\h, C} &
 	F_{\f1} \colon \Op_\h^{\f1} &\leftrightarrows \Op_\h \colon U^{\f1}.
\end{align}

\begin{lemma}
  \label{lemma:iota-1-is-cofibration}
  The natural map $\iota \colon F_{0}^{-1}W_{-1}^C\calO \to W_{0}^C\calO$ is a cofibration in $\Op^C_0$.
\end{lemma}
\begin{proof}
  If $y \in C_{-1}$, then $F_{0}^{-1}W_{-1}^C\calO(y) \to W^C_0\calO(y)$ is an isomorphism.  If $y \in C_{0}-C_{-1}$, there is no weighted forest $f\colon x \to y$ with $x \in C_{-1}$.  Thus $F_{0}^{-1}W_{-1}^C\calO(y) = \emptyset$ and $W^C_0\calO(y) = \calO(y)$.   Thus $\iota$ is certainly a cofibration of collections.  If $f \colon x \to y$ has at least one internal edge then $x \in C_0$, so $\iota(x)$ is an isomorphism.  Thus $\iota$ is a cofibration of operads in $\Op^C_0$.
\end{proof}

\begin{lemma}
  \label{lemma:FWCO-to-WCO-cofibration-of-weight-k+1-operads}
  For any $k\geq 0$ the natural map $F^k_{k+1}W^C_{k}\calO \to W^{C}_{k+1}\calO$ is a cofibration in $\Op^C_{k+1}$
\end{lemma}
\begin{proof}
  The argument from lemma~\ref{lemma:FWO-to-WO-cofibration-of-weight-k-operads} works in this case.
\end{proof}

\begin{corollary}
	\label{corollary:FWhf1O-to-WhO-is-cofibration-of-Oph-operads}
	The natural map $F_{\f1} W_{\h}^{\f1} \calO \to W_{\h} \calO$ is a cofibration in $\Op_{\h}$.
\end{corollary}
\begin{proof}
  We have $W_{\h}^{\f1} \calO = F^{\f1,-1}_{\h,C} W_{-1}^C\calO$ and $F^\omega_{C,\h}W^C_\omega\calO = W_\h \calO$. By lemmas~\ref{lemma:iota-1-is-cofibration} and \ref{lemma:FWCO-to-WCO-cofibration-of-weight-k+1-operads}, $F^{-1}_\omega W_{-1}^C\calO \to W^C_{\omega}\calO$ is a cofibration in $C_\omega$.  Apply $F^\omega_{\h,C}$ and use the fact that $F^\omega_{\h,C} F^{-1}_\omega = F_{\f1}F^{\f1,-1}_{\h,C}$ to get $F_{\f1}W_{\h}^{\f1}\calO \to W_{\h}\calO$.
\end{proof}

The main example of an $\{\f, \h\}$-colored operad is $\swcheese$, where $\f$ stands for \emph{full disc} and $\h$ stands for \emph{half disc}. Fix a dimension $d \geq 1$.  Let $D_\f$ denote the closed unit disc inside $\reals^{d}$, and let $D_\h$ denote the closed unit half-disc $\{ p \in D_\f \mid p_d \geq 0\}$.  For any $K$-colored set $(I,\col_I)$ and any $i \in I$ let $D_i = D_{\col_I(i)}$.  Finally, put $D(I)=\coprod_{i\in I} D_i$.

\begin{definition}
  \label{definition:swcheese-operad}
  Given a $K=\{\f,\h\}$-colored young forest $x$, let $\swcheese(x)$ denote the set of maps $\alpha \colon D(I_x) \to D(J_x)$ such that
  \begin{itemize}
    \item
      For each $i \in I_x$ the restriction of $\alpha$ to $D_i$ lands in $D_{x(i)}$.
    \item
      For each $i \in I_x$ of color $\f$ there is an $r(i) > 0$ and $c(i) \in \reals^d$ such that $\alpha(p) = r(i)p + c(i)$ for all $p \in D(i)$.
    \item
      For each $i \in I_x$ of color $\h$, $x(i)$ has color $\h$ and there is an $r(i) > 0$ and $c(i) \in \reals^{d-1}\times\{0\}$ such that $\alpha(p) = r(i)p + c(i)$ for all $p \in D(i)$.
    \item
      $\alpha$ is an embedding of $D(I_x)$ into $D(J_x)$.
  \end{itemize}
  Let $(n,m)$ denote the $K$-colored set $\{1,\ldots, n+m\}$ where $1\leq i \leq n$ has color $\f$ and $n+1 \leq i \leq n+m$ has color $\h$.  Suppose $J_x = \{j\}$, then an isomorphism $I_x \to (n,m)$ defines an embedding of $\swcheese(x)$ into $\reals^{N}$, where $N = n + nd + m + m(d-1)$.  We endow $\swcheese(x)$ with the topology induced from such an embedding.  If $J_x$ is not a singleton then $x = \coprod_j x|x\inv(j)$ and $\swcheese(x) = \prod_{j\in J_x} \swcheese(x|x\inv(j))$, and we take the product topology on $\swcheese(x)$.

  If $f \colon x \to y$ is a $K$-colored forest, the map $\swcheese(f) \colon \swcheese(x) \to \swcheese(y)$ is defined to be
  \[
    \swcheese(f)(\alpha)(p) = (D(f) \alpha)^k D(f)(p)
  \]
  where $\alpha \in \swcheese(x)$, $p \in D_i$, $i \in I_y$, and $[f|x](i) = (fx)^kf(i)$.  The isomorphism $f \colon I_y \sqcup J_x \to J_y \sqcup I_x$ induces the isomorphism $D(f) \colon D(I_y) \sqcup D(J_x) \to D(J_y) \sqcup D(I_x)$.
\end{definition}

In \cite{thomas-kontsevich-swiss-cheese}, $\schinf$ denotes $F_{\f1} W_{\h}^{\f1} \swcheese$ and $\sch$ denotes $W_\h \swcheese$.  In this special case corollary~\ref{corollary:FWhf1O-to-WhO-is-cofibration-of-Oph-operads} gives theorem~\ref{theorem:swcheese-generated-in-degree-0-1}, which is an essential technical result of that paper.   Informally, theorem~\ref{theorem:swcheese-generated-in-degree-0-1} combined with theorem~\ref{theorem:schinf-sch-weak-equivalence} means that the swiss cheese operad is generated in degrees 0 and 1.

\begin{theorem}
	\label{theorem:swcheese-generated-in-degree-0-1}
  The natural map $\schinf \to \sch$ is a cofibration in $\Op_\h$.
\end{theorem}

\subsection{Weak equivalence proof}
\label{section:weak-equivalence}
 This section contains a proof of
 \begin{theorem}\label{theorem:schinf-sch-weak-equivalence}
  The natural map $\schinf \to \sch$ is a weak equivalence of operads in $\Op_\h$.
\end{theorem}
 The idea of the proof is to consider the maps
$p_1 \colon \schinf(n,m) \to \schinf(n-1,m)$ and $p \colon\sch(n,m) \to
\sch(n-1,m)$ given by forgetting the $n^{th}$ disc.  By induction, we can
suppose $\schinf(n-1,m) \to \sch(n-1,m)$ is a weak equivalence.  We continue the induction by showing  that $p_1\inv(\alpha) \to p\inv(\alpha)$ is a weak equivalence for
every $\alpha \in \schinf(n-1,m)$.

To make the computation of $p_1\inv(\alpha)$ and $p\inv(\alpha)$ accessible, we will collapse the $n^{th}$ disc of $\alpha \in \sch(n,m)$ to a point.  Our goal in the next section is to make this precise.

\subsubsection{Defining \texorpdfstring{$\swcheesebullet(k,l|n,m)$}{ SCbullet(k,l|n,m)}}

When we collapse the $n^{th}$ disc of $\alpha \in \swcheese^\h(n,m)$ to its center, we think of the result $\hat{\alpha}$ as living in a four-colored operad which we denote by $\swcheesebullet$.  We add the colors $\e_\smallbullet$ and $\h_\smallbullet$.  Let $K_\smallbullet = \{ \e_\smallbullet, \h_\smallbullet, \e, \h\}$ be the set of colors for this new operad.  The color $\e_\smallbullet$ stands for collapsed full disc.  It is convenient to also allow a collapsed half disc, which we color with $\h_\smallbullet$.  Let $(k,l|n,m)$ denote the $K_\smallbullet$-colored finite set with $(k,l,n,m)$ elements of color $(\e_\smallbullet, \h_\smallbullet, \e, \h)$.  Let $\rfor^{ K_\smallbullet}_{\h}$ denote the full sub category of $\rfor^{K_{\smallbullet}}$ with objects isomorphic to disjoint unions of the young forests
\begin{equation}\label{equation:smallbullet-connected-young-forests}
  \hspace{-0.9em}(0,0|n,m) \to \{\h\} \quad (1,0|n,m) \to \{\h\} \quad (0,1|n,m) \to \{\h\} \quad (1,0|0,0) \to \{\h_\smallbullet\}
\end{equation}
To define $\swcheesebullet\colon \rfor^{ K_\smallbullet}_{\h} \to \topl$ we need the notion of the geometric realization of $\beta\in \swcheese^\h(n, m)$.
\begin{definition}
\label{definition:geometric-realization}
 Given $\beta \in \swcheese^\h(n,m)$, let $\abs{\beta}$ be its \emph{geometric realization}.  This is the subset of $\mathbb R^d$ given by deleting the open discs and half-discs of $\beta$ from the closed unit half-disc.  More precisely, if $\bar{D}^d_+$ is the closed unit half-disc in $\mathbb R^d$, $\{(D^d_\e)_j\}_{j=1}^m$ are the open discs of $\beta$, and  $\{(D^d_\h)_i\}_{i=1}^{n}$ are the open half-discs of $\beta$ considered as open discs in $\mathbb R^d$ whose center lies in $\mathbb R^{d-1}$, then
\[
 \abs{\beta} = \bar D^d_+ - \left(\left(\bigcup_{i=1}^{n} (D^d_\h)_i \right) \cup \left(\bigcup_{\smash{j}=1}^m (D^d_\e)_j \right) \right).
\]
Let $\partial_\h \!\abs{\beta} \coloneqq \partial(\bar D^d_+ - (\cup_i (D^d_\h)_i))$ be the \emph{$\h$-colored boundary} of $\abs{\beta}$.  Let $\partial_{\rt}(\abs{beta})$ be the upper hemisphere $S^{d-1}_+ \subset \partial \bar D^d_+$ and let $\partial_{i}\abs{\beta}$ be the upper hemisphere of $\partial(D^d_\h)_i$ for $1 \leq i \leq n$.
%In addition, set $\partial^{\out}_i\abs{\beta}$ to be the closed subset of $\partial^{\out}\abs{\beta}$ given by the boundary of the $i^{th}$ half-disc, $1\leq i \leq m$.  Set $\partial^{\out}_{rt}\abs{\beta}$ to be the closed subset of $\partial^{out} \abs{\beta}$ given by the upper hemisphere of the closed unit half-disc $\bar D^d_+$.
\end{definition}
Now we can set
\begin{align*}
\swcheesebullet^\h(0,0|n,m) &= \swcheese^\h(n,m) \\
\swcheesebullet^\h(1,0|n,m) &= \{(\alpha, q) \mid \alpha \in \swcheese^\h(n,m), q \in \abs{\alpha}\}  \\
\swcheesebullet^\h(0,1|n,m) &=\{ (\alpha, q) \mid \alpha \in \swcheese^\h(n,m), q \in \abs{\alpha}\cap \mathbb R^{d-1} \}\\
\swcheesebullet^{\h_\smallbullet}(k,l|n,m) &= \ast
\end{align*}
We think of the point $q \in \abs{\alpha}$ as a collapsed disc and the point $q \in \abs{\alpha}\cap \mathbb R^{d-1}$ as a collapsed half-disc.
Composition in $\swcheesebullet$ takes place in the half-discs and collapsed half-discs only. The discs play no part in composition.  However the collapsed half-discs and collapsed discs only play a part in composition when we plug a collapsed disc into a collapsed half-disc.  The result is a collapsed disc which happens to live on the boundary of the geometric realization.

\begin{figure}[ht]
\label{swcheesehat-def}
 \centering
 \includegraphics[width=\textwidth]{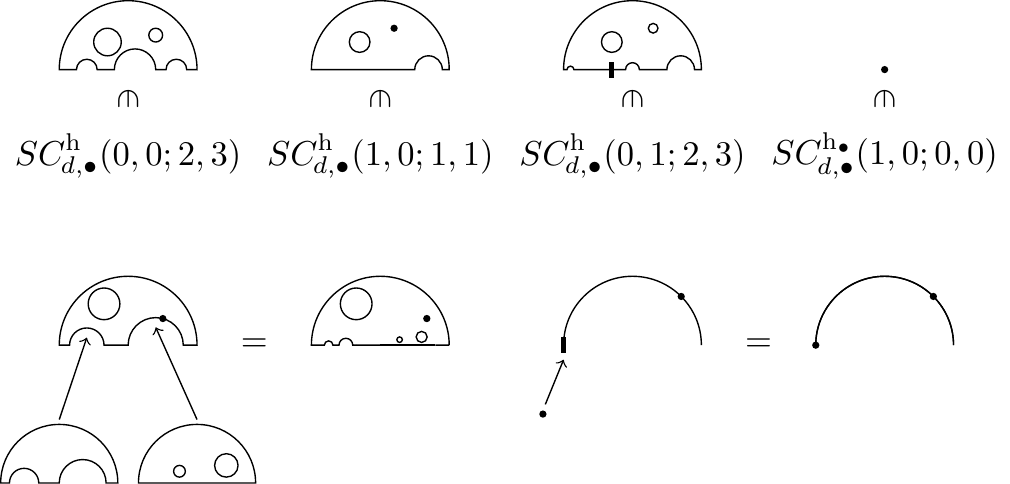}
\caption[The 3-colored operad $\swcheese^\h$]{The collapsed discs are denoted by dots and the collapsed half-discs by tick marks.  Collapsed discs are color $\e_\smallbullet$ input edges and collapsed half-discs are color $\h_\smallbullet$ input edges.  To keep the collapsed discs and half-discs from coinciding, we only allow one or the other in any composition.  Composition in $\swcheesebullet$ takes place only in the half-discs and collapsed half discs.  The only composition we can do in a collapsed half-disc is given by plugging in a collapsed disc.  The result is a collapsed disc replacing the collapsed half-disc.}
\end{figure}
%$\operads_{\widehat K}^{\h\hat \h;\hat \e + \hat\h +\e \leq 1}$ be the category consisting of operads $\mathcal O \in \operads_{\widehat K}^{\h\hat\h;\hat \e+\hat\h \leq 1}$ which further satisfy the condition
%\[
%\mathcal O^\h(k,l;n,m) = \emptyset \text{ if } k +l +n \geq 2 .
%\]
%There is a commutative diagram
%\[
%\begin{tikzpicture}
%  \matrix (m) [matrix of math nodes, row sep={5ex}, column sep={5ex}, text height=2.4ex, text depth=0.65ex]
%{ \operads_{ K}^{\h1}  & \operads_{K}^{\h}  \\
% \operads_{\widehat K}^{\h\hat\h;\hat\e+\hat\h+\e\leq 1} & \operads_{\widehat K}^{\h\hat\h;\hat\e+\hat\h\leq 1} \\};
%\path[->]
% (m-1-1) edge node[above, map name] {$F$} (m-1-2)
% (m-2-1) edge node[above, map name] {$F$} (m-2-2)
% (m-1-1) edge (m-2-1)
% (m-1-2) edge (m-2-2);
% \end{tikzpicture}
%\]
%where each $F$ is the left adjoint to the obvious forgetful functor.  The vertical maps consider an operad in $\operads_{K}^{\h1}$ or $\operads_K^\h$ to be empty whenever there is an input or output of color $\hat\e$ or $\hat\h$.
\begin{definition}\label{definition:e-leq-1-forest-categories}
Let $\rfor^{K_\smallbullet}_{\e\leq 1;\h}$ denote the full sub category of $\rfor^{K_\smallbullet}_\h$ given by disjoint unions of the young forests from \ref{equation:smallbullet-connected-young-forests} with $n \leq 1$. In the same way, let $\rfor^{K}_{\e\leq 1;\h}$ denote the full sub category of $\rfor^K_\h$ given by disjoint unions of forests from \ref{equation:smallbullet-connected-young-forests} with $k = l =0$ and $n \leq 1$.   We write $x \leq 1$ if $x$ is a $K_\smallbullet$-colored young forest in $\rfor^{K_\smallbullet}_{\e\leq 1;\h}$ and we alsow write $x \leq 1$ if $x$ is a $K$-colored young forest in $\rfor^{K}_{\e\leq 1;\h}$.  Let $\swcheesebullet^{\h 1}$ denote the restriction of $\swcheesebullet$ to $\rfor^{K_\smallbullet}_{\e\leq 1;\h}.$
\end{definition}
%The map $K_\smallbullet \to K$ given by $\e_\smallbullet,\e \mapsto \e$ and $\h_\smallbullet,\h\mapsto \h$ induces a functor $\rfor^{K_{\smallbullet}}_\h \to \rfor^K_{\h}$, which allows us to consider $\swcheese^\h$ as an operad of the same type as $\swcheesebullet$.  In the same way there is a functor $\rfor^{K_\smallbullet}_{\e\leq 1;\h} \to \rfor^{K}_{(1,\infty)}$ (see definition~\ref{definition:forests-with-genus}) allowing us to consider $\swcheese^{\h 1}$ as an operad of the same type as $\swcheesebullet^{\h 1}$.
%We can consider $\swcheese^\h(*,*)$ as a $\widehat K$ colored operad by setting
%\[
%\swcheese^\h(k,l;n,m) = \swcheese^\h(k+n, l+m) \text{ and } \swcheese^{\hat h}(1,0;0,0) = \swcheese^\h(1,0).
%\]

Let $\schbullet$ and $\schonebullet$ denote the $\W$ construction applied to the four-colored operads $\swcheesebullet$ and $\swcheesebullet^{\h 1}$.  Let $F$ denote Kan extension along $\rfor^{K_\smallbullet}_{\e\leq 1;\h} \to \rfor^{K_\smallbullet}_{\h}$.  Consider the commutative diagram of topological spaces where the horizontal arrows do not assemble to operad maps,
\[
\begin{tikzpicture}
  \matrix (m) [matrix of math nodes, row sep={2ex}, column sep={3ex}, text height=1.5ex, text depth=0.25ex]
{ \schinf(n+1, m)  &   F(\schonebullet)(1,0| n, m) & \schinf(n, m)\\
 \sch(n+1, m) & \schbullet(1,0| n, m) & \sch(n, m).
 \\};
\path[->]
 (m-1-1) edge (m-2-1) edge node[above, map name] {$\sim$} (m-1-2)
 (m-1-2) edge (m-2-2) edge node[above, map name] {$p_1$} (m-1-3)
 (m-2-1) edge node[above, map name] {$\sim$} (m-2-2)
 (m-2-2) edge node[above, map name] {$p$} (m-2-3)
 (m-1-3) edge node[right, map name] {$\iota$} (m-2-3);
 \end{tikzpicture}
\]
The maps $p_1$ and $p$ delete the collapsed disc and, if necessary, a left over collapsed half-disc.
By induction on $n$ we assume the right vertical arrow is an equivalence. We will show that for each $\alpha\in \schinf(n, m)$ the inclusion $p_1^{-1}(\alpha)\rightarrow p^{-1}(\iota\alpha)$ is an equivalence. Then by the long exact sequence of homotopy groups we conclude that the middle vertical arrow is an equivalence. The top left and bottom right maps collapse the $n^{th}$ full disc.  One can show that these are equivalences. We conclude that the left vertical arrow is also an equivalence.  This will prove theorem~\ref{theorem:schinf-sch-weak-equivalence}.

\subsubsection{Computing \texorpdfstring{$p\inv(\iota\alpha)$ and $p_1\inv(\alpha)$}{p-1(iota alpha) and p1-1(alpha)}.}

We have shown that the proof rests on the following proposition~\ref{proposition:fiber-equivalence}.  This section is dedicated to the proof of this proposition.
\begin{proposition}
\label{proposition:fiber-equivalence}
 Fix $\alpha \in \schinf(n,m)$.  The inclusion of the fiber $p_1\inv(\alpha)$ into the fiber $p\inv(\iota \alpha)$ is a weak equivalence.
\end{proposition}

%The proof for this theorem will use some observations about the $\W$ construction.  We will compute $p\inv(\alpha)$ and $p_1\inv(\alpha)$ as homotopy colimits of functors $F$ and $F_1$.  We can more easily examine these homotopy colimits and see by inspection that they are equivalent.

%Define $\trees(k,l;n, m)$ to be the category whose objects are trees with $(k,l;n, m)$ input edges of color $(\hat\e,\hat\h;\e,\h)$, and whose morphisms are given by collapsing internal edges or inserting a vertex along an edge.  The category $\trees(n, m)$ is defined in a similar way. There is a functor $p:\trees(1,0; n , m)\rightarrow\trees(n, m)$, which deletes the single input edge of color $\hat\e$ and the subsequent edge of color $\hat \h$, if it is there.

\begin{definition}\label{definition:three-p's}
Define $p \colon \rfor^{K_\smallbullet}_\h \to \rfor^K_\h$ by sending the $K_\smallbullet$-colored young forest $x \colon I_x \to J_x$ to the $K$-colored forest $px$ with
\[
	I_{px} = I_x\smallsetminus(I_x)_{\e_\smallbullet, \h_\smallbullet} \quad J_{px} = J_x \smallsetminus (J_x)_{\e_\smallbullet, \h_\smallbullet},
\]
where $I_{\e_\smallbullet, \h_\smallbullet}$ is the $\e_\smallbullet$ and $\h_\smallbullet$-colored portion of the $K_\smallbullet$-colored set $I$.  In \ref{equation:smallbullet-connected-young-forests} we see that we must have $x(I_{px})\subset J_{px}$ so that we can define $px$ as the restriction of $x$ to $I_{px}$.  Observe that $p(1,0|n,m) = (n,m)$.  If $f\colon y \to x$ is a forest, then $pf \colon py \to px$ is defined using $f$.  Since $f$ preserves the colorings $pf$ is indeed a forest from $py$ to $px$.  If $f$ is in $\rfor^{K_\smallbullet}_{\e\leq 1;\h}$, then $pf$ is a morphism in $\rfor^K_{\e\leq 1; \h}$.

If $\beta \in \swcheesebullet(z)$ for a $K_\smallbullet$-colored young forest $z$, then we get $p\beta \in \swcheese(pz)$.  To define $p\beta$ write $\beta = (\beta_j)_{j\in J_z}$ where $\beta_j \in \swcheesebullet(z\inv(j))$.  Each $\beta_j$ is of the form $(\gamma_j, q_j)$ with $q_j \in \abs{\gamma_j}$ or of the form $\gamma_j \in \swcheese(z\inv(j))$.  Set $p\beta = (\gamma_j)_{j \in J_{pz}}$.

If $t \in W(f)$ and $f \in \rfor^{K_\smallbullet}_\h$, then $E(pf) \subset E(f)$ and $pt \in W(pf)$ is defined to be the pullback of $t \colon E(f) \to [0,\infty]$.
\end{definition}

Combining the colimits defining the $W$ construction and the left adjoint $ \Operads(\Collleqone) \to \Operads(\Coll)$ we get
\[
	\schinf(n,m) = \Bigg(\coprod_{\substack{g \colon z\to y\\ f \colon y \to (n,m)}}W(g) \times \swcheese(z) \Bigg)\modsim,
\]
where $y \leq 1$ (definition~\ref{definition:e-leq-1-forest-categories}) and the same relations hold as in equation~\ref{equation:FkWO(z)}.  If $\alpha \in \schinf(n,m)$ is represented by $(f,g, t, \tilde\alpha)$ where $f \colon y \to (n,m)$, $g \colon z_\alpha \to y$, $t\in W(g)$, and $\tilde\alpha \in \swcheese(z)$, then $\iota \alpha \in \sch(n,m)$ is represented by $(fg, W_\infty(f)t,\tilde\alpha)$.  Let $T_\alpha = fg\colon z_\alpha \to (n,m)$ and $t_\alpha=W_\infty(f)t$.  Without loss of generality, we may assume $t_\alpha(i) > 0$ for every $i \in E(T_\alpha)$ and that $\tilde\alpha(j) \neq \id_{\swcheese}$ for any $j \in J_{z_\alpha}$.

\begin{definition}\label{definition:W-alpha}
Let $\trees(1,0|n,m)$ denote the over category $\rfor^{K_\smallbullet}_{\h,/(1,0|n,m)}$ %over the young $K_\smallbullet$-colored forest $(k,l|n,m)\to\{ \h \}$.
%The objects of $\trees(k,l|n,m)$ are trees whose input vertices are labeled by $(k,l|n,m)$.  The morphisms are given by compositions of morphisms which collapse an internal edge or insert a unary vertex along an edge.
Let $\trees(n,m) = \rfor^K_{\h, /(n,m)}$.  Let $p \colon \trees(1,0|n,m) \to \trees(n,m)$ denote the functor induced by $p$ from definition~\ref{definition:three-p's}.

 Note that $T_\alpha \in \trees(n,m)$.  Let $(S, \nu) \in \trees(1,0|n,m)_{/T_\alpha}$ where $S\colon x \to (1,0|n,m)$ is any $K_\smallbullet$-colored tree and $\nu \colon z_\alpha \to px$ is a forest such that $(pS)\nu = T_\alpha$. Define functors $W_\alpha \colon \trees(1,0|n,m)_{/T_\alpha}^{op} \to \topl$ and $\swisscheese_\alpha \colon \trees(1,0|n,m)_{/T_\alpha} \to \topl$ via the pullbacks
\begin{equation}\label{diagram:swisscheesealpha-and-Walpha}
    \begin{tikzpicture}[numbered picture]
      \matrix (m) [matrix of math nodes, row sep=2ex, column sep=1em, text height=1.5ex, text depth=0.5ex]
      {
        \swisscheese_\alpha(S) &             & \swcheesebullet(x) & & W_\alpha(S) & W(S)  & W(pS) \\
        \ast                   &|[xshift=-3em]|\swcheese(z) & \swcheese(px)      & & \ast        &       & W(T_\alpha).
      \\};
      \foreach \source/\target in {1-1/1-3, 1-1/2-1, 2-1/2-2, 2-2/2-3, 1-3/2-3, 1-5/1-6, 1-5/2-5, 2-5/2-7, 1-6/1-7, 1-7/2-7}
      {\path[->] (m-\source) edge (m-\target);}
      \foreach \source/\target/\pos/\maplabel in {1-3/2-3/right/p, 2-1/2-2/above/\tilde\alpha, 2-2/2-3/above/\swcheese(\nu), 1-6/1-7/above/p, 1-7/2-7/right/W_\Sigma(\nu), 2-5/2-7/above/t_\alpha}
      {\path (m-\source) edge node[map name, \pos] {$\maplabel$}
      (m-\target);}
    \end{tikzpicture}
  \end{equation}
\end{definition}

We want to replace $\trees(1,0|n,m)$ by a much smaller category.  First we need the wedge operation on forests.
\begin{definition}
  \label{definition:wedging-forests}
  Let $f \colon x \to y$ be a $K_f$-colored forest and let $g \colon z \to w$ be an $K_g$-colored forest for some finite sets $K_f, K_g$.  Let $\tau \colon J_w \to J_x$ be any map.  Define $x \vee_\tau z$ to be the young $K_f \sqcup K_g$-colored forest $(x,\tau, z) \colon I_x \sqcup J_w \sqcup I_z \to J_x \sqcup J_z$ and define $y \vee_\tau w$ to be the young forest $(y,(xf)^\infty \tau w) \colon I_y \sqcup I_w \to J_y$.  Finally, set $f\vee_\tau g \colon x \vee_\tau z \to y \vee_\tau w$ to be the forest $(f,g,f,\tau)\colon I_y \sqcup I_w \sqcup J_x \sqcup J_z \to I_x \sqcup J_w \sqcup I_z \sqcup J_y$.
\end{definition}

\begin{definition}
  \label{definition:trees-alpha}
  Let $\Gamma_0$ be the tree with with no internal vertices and a single input vertex of color $\e_\smallbullet$.  Let $\Gamma_1$ be the tree with a single internal vertex of color $\h_\smallbullet$ and a single input vertex of color $\e_\smallbullet$.

  For any edge $i \in \edges(T_\alpha)_\h$ define $\nu(i) \colon T_\alpha \to T_\alpha(i)$ to be the morphism in $\trees(n,m)$ which inserts a unary vertex along $i$. Call this new vertex $i_v$.   Let $S_{i,k} = T_\alpha(i) \vee_{i_v} \Gamma_k$.  For any internal vertex $j \in J_{z_\alpha}$ let $S_{j,k} = T_\alpha \vee_{j} \Gamma_k$.  Note that $pS_{i,k} = T_\alpha(i)$ and $pS_{j,k} = T_\alpha$.

  Let $\trees_\alpha$ be the full subcategory of $\trees(1,0|n,m)_{/T_\alpha}$ given by the objects $S_{i,k}=(S_{i,k},\nu(i))$ and $S_{j,k}=(S_{j,k},\id_{T_\alpha})$ where $i \in (I_{z_\alpha})_\h\sqcup\{ \rt \}$, $j \in J_{z_\alpha}$ and $k \in \{ 0,1 \}$.
\end{definition}

\begin{remark}
The advantage of $T_\alpha$ is that it is easy to understand and computes the space $p\inv\alpha$ (lemma~\ref{lemma:p-inv-alpha-coend}).  There is a unique morphism $S_{\ell, 1}$ to $S_{\ell, 0}$ for every $\ell$ and unique morphisms $S_{i,k} \to S_{T_\alpha\inv(i),k}$ and $S_{i,k} \to S_{z_\alpha(i),k}$.   See figure \ref{figure:trees-over-Talpha} for an illustration.
\end{remark}
   %Let $\trees_\alpha$ be the full sub category of $\trees(1,0|n,m)_{/T_\alpha}$ given  by $(S, \nu)$ where either $\nu = \id_{z_\alpha}$ or $\nu$ is the insertion of the vertex $j \in J_{px} \subset J_x$ such that $x \inv(j)_{\e_\smallbullet,\h_\smallbullet}\neq \emptyset$.

\begin{lemma}
\label{lemma:p-inv-alpha-coend}
 The fiber $p\inv(\alpha)$ is given by the coend
\[
 W_\alpha \otimes_{\trees_{\alpha}} \swisscheese_\alpha.
\]
\end{lemma}
\begin{proof}
Let $\gamma = [S, s, \tilde\gamma] \in \schbullet(1,0|n,m)$ where $S \colon x \to (1,0|n,m)$ is a forest in $\rfor^{K_\smallbullet}_\h$, $s \in W(S)$, and $\tilde\gamma \in \swcheesebullet(x)$.  Let us assume that $\tilde\gamma(j)\neq \id$ for all $j \in J_x$ and $s(i) > 0$ for all $i \in E(S)$.  Observe that $p\gamma \in \sch(n,m)$ is given by $[pS,ps,p\tilde\gamma]$.  If $p\gamma=\alpha$ there must be some $\nu \colon T_\alpha \to pS$ in $\trees(n,m)$ such that $\swcheese(\nu)\tilde\alpha= p\tilde\gamma$ and $W(\nu)ps = t_\alpha$.  The condition $t_\alpha(i) >0$ for all $i\in E(T_\alpha)$ implies that $t_\alpha\neq W_\Sigma(\nu)(t')$ for any $t'$ and any $\nu$ which collapses any edges.  Moreover the condition $\tilde\gamma(j) \neq \id$ for all $j$ implies that $p\tilde\gamma(j) \neq \id$ for all $j\in J_x$ such that $x\inv(j)_{\e_\smallbullet, \h_\smallbullet}$ is not empty.  We conclude that either $\nu =\id$ or $\nu$ is the insertion of the unique unary (in $pS$, not in $S$) vertex $j$ such that $x\inv(j)_{\e_\smallbullet, \h_\smallbullet}$ is not empty.  In the former case we must have $S= S_{j,k}$ for some vertex $j \in J_x$ and some $k \in \{  0,1 \}$  In the latter case we have $S=S_{i,k}$ for some edge $i$ of $T_\alpha$ and some $k$. This defines the map $p\inv(\alpha) \to W_\alpha \otimes_{\trees_\alpha} \swisscheese_\alpha$.  The map in the other direction is clear and the verification that they are inverses is left to the reader.
%
%On the other hand, since $S$ has no vertices labeled with the identity by $\tilde \gamma$, the tree $pS$ can only differ from $T_\alpha$ by a possible insertion
%
%If $g \colon x' \to x$ is any $K_\smallbullet$-colored forest and $\tilde\gamma=\swcheesebullet(\tilde\beta)$ for some $\tilde\beta \in \swcheesebullet(x')$ then $\gamma= [Sg,W(g)s,\tilde\beta]$.
%
%Given any tree $T$ the space $W_\alpha(T)\times\swisscheese_\alpha(T)$ is the space of all edge and vertex labels of $T$ which,  after deleting the marked point corresponding to the $\hat\e$ input edge, give the representative $(t_\alpha, \tilde \alpha) \in W(T) \times \swcheese(T)$ of $\alpha$ up to a possible extra identity vertex. The relations that hold in the coend are exactly the relations in the $\W$ construction which allow us to delete edges labeled $0$ and vertices labeled with the identity.
\end{proof}

\begin{figure}[htbp]
\label{figure:trees-over-Talpha}
\centering
\includegraphics[width=\textwidth]{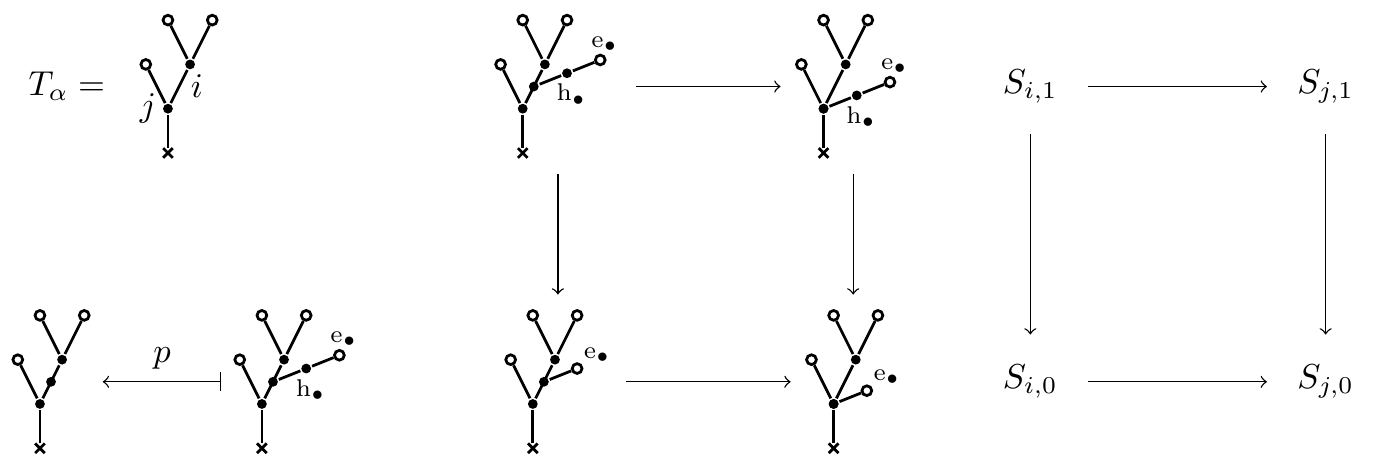}
 \caption[The category $\trees_{\alpha}$]{The edge $i$ and vertex $j$ of $T_\alpha$ give a commutative square in $\trees_{\alpha}$.  The input vertices are circles.  The output vertex ends in an $\mathsf x$.  The internal vertices are filled dots.  The input and internal vertices of $\Gamma_{0}$ and $\Gamma_{1}$ are labeled with their colors.  In addition, the image of $S_{i,1}$ under the functor $p$ is shown.  This makes it clear that the map $T_\alpha \to pS_{i,1}$ is given by inserting a single vertex.}
\end{figure}

In diagram~\ref{diagram:swisscheese-alpha-double-cube} we have $\h$-colored edges $i_1, i_2$ of $T_\alpha$ with $z_\alpha(i_1) = j = T_\alpha\inv(i_2)$.  Thus we get the commutative diagram on the left.  The image of this diagram under $\swisscheese_\alpha$ is shown on the right.

\begin{equation}\label{diagram:swisscheese-alpha-double-cube}
\begin{tikzpicture}[numbered picture]
      \matrix (m) [matrix of math nodes, row sep={2ex}, column sep={0.9em},
         text height=1.5ex, text depth=0.25ex]
      {
       S_{i_1,1} & S_{j,1} & S_{i_2, 1} & \abs{\id_{\h}} \cap \mathbb R^{d-1} & \abs{\tilde \alpha(j)}\cap  \mathbb R^{d-1} &  \abs{\id_{\h}}\cap \mathbb R^{d-1} \hspace{-2em}\\
       S_{i_1,0} & S_{j,0} & S_{i_2,0} &  \abs{\id_{\h}} & \abs{\tilde \alpha(j)} &  \abs{\id_{\h}}
      \\};
      \foreach \source/\target in {1-1/1-2, 1-1/2-1, 2-1/2-2, 1-2/2-2, 1-3/1-2, 1-3/2-3, 2-3/2-2, 1-4/1-5, 1-4/2-4, 2-4/2-5, 1-5/2-5, 1-6/1-5, 1-6/2-6, 2-6/2-5}
      {\path[->] (m-\source) edge (m-\target);}
      \foreach \source/\target/\pos/\maplabel in {}
      {\path (m-\source) edge node[map name, \pos] {$\maplabel$}
      (m-\target);}
   \end{tikzpicture}
\end{equation}
The geometric realization of the identity $\id_\h$ is just $S^{d-1}_+$, the top half of the $(d-1)$-sphere.  The input of $\tilde\alpha(j)$ corresponding to $i_1$ is a half disc and the map $\abs{\id_\h} \to \abs{\tilde\alpha(j)}$ corresponding to $S_{i_1, 0} \to S_{j,0}$ is just $\partial_{i_1}\abs{\tilde\alpha(j)} \to \abs{\tilde\alpha_j}$ (see definition~\ref{definition:geometric-realization}).  On the other hand the image of $S_{i_2, 0} \to S_{j,0}$ is the inclusion of the output boundary $\partial_\rt\abs{\tilde\alpha(j)} \to \abs{\tilde\alpha(j)}$.

\begin{definition}
  \label{definition:the-three-edges-of-S}
  Let $\epsilon_\smallbullet \in E(S_{\ell,1})$ be the unique internal edge of color $\h_\smallbullet$.  If $i \in \edges(T_\alpha)$, let $i_v$ denote the vertex inserted by $\nu \colon T_\alpha \to pS_{i,k}$.  Let $i_\inverts$ and $i_\outverts$ respectively denote the incoming and outgoing edges of $i_v$ considered as internal edges of $S_{i,k}$.  For any object $S_{\ell,k}$ of $\trees_\alpha$, let $E_\alpha(S_{\ell,k}) = \{ \epsilon_\smallbullet \}^k \sqcup (\{ i_\inverts, i_\outverts \} \cap E(S_{\ell,k}))$.  This defines a functor $E_\alpha \colon \trees_\alpha^{op} \to \mathrm{Set}$.
  %For any internal edge $i$ of $T_\alpha$ let $E_\alpha(S_{i,1}) = \{ \epsilon_\smallbullet, i_\inverts, i_\outverts \}$ and $E_\alpha(S_{i,0}) = \{ i_\inverts, i_\outverts \}$.  For any vertex $j$ let $E_\alpha(S_{j,1}) = \{ \epsilon_\smallbullet \}$ and $E_{\alpha}(S_{j,0}) = \emptyset$.  For any input edge $i$ let $E_\alpha(S_{i,1}) = \{ \epsilon_{\smallbullet},i_\outverts \}$ and $E_\alpha(S_{i,0}) = \{ i_\outverts \}$.  Finally, for $i \in \rt(T_\alpha)$ set $E_\alpha(S_{i,1}) = \{ \epsilon_{\smallbullet}, i_\inverts \}$ and $E_\alpha(S_{i,0})=\{ i_{\inverts} \}$.
\end{definition}

The image under $W_\alpha$ of the square in diagram~\ref{diagram:swisscheese-alpha-double-cube} is in diagram~\ref{diagram:W-alpha-double-cube}.
\begin{equation}\label{diagram:W-alpha-double-cube}
\begin{tikzpicture}[numbered picture]
      \matrix (m) [matrix of math nodes, row sep={2.2ex}, column sep={3em},
         text height=1.7ex, text depth=0.25ex]
      {
       \lbrack0,\infty\rbrack^2 & \lbrack0,\infty\rbrack & \lbrack0,\infty\rbrack^2 \\
       \lbrack0,\infty\rbrack   & \ast       & \lbrack0,\infty\rbrack
      \\};
      \foreach \source/\target in {1-2/1-1, 1-2/1-3, 2-1/1-1, 2-3/1-3, 2-2/1-2, 2-2/2-1, 2-2/2-3}
      {\path[->] (m-\source) edge (m-\target);}
      \def\comma{,}
      \foreach \source/\target/\pos/\maplabel in {1-2/1-1/above/(\id\comma 0), 1-2/1-3/above/(\id\comma\infty), 2-1/1-1/left/(0\comma\id), 2-3/1-3/right/(0\comma\id), 2-2/1-2/right/0, 2-2/2-1/above/0, 2-2/2-3/above/\infty}
      {\path (m-\source) edge node[map name, \pos] {$\maplabel$}
      (m-\target);}
   \end{tikzpicture}
\end{equation}
%More precisely, consider $S_{i,1}$ for any $i \in \edges(T_\alpha)$ and suppose $s \in W_\alpha(S_{i,1})$.  There are three internal edges in $E(S_{i,1})$ on which we have some choice for the value of $s$.  Call these edges $i_\smallbullet, i_\inverts$, and $i_\outverts$.  These are all the edges adjacent to the vertex $i_v$ in $S_{i,1} = T(i)\vee_{i_v} \Gamma_1$.
More precisely,
\begin{equation}
  \label{equation:W-alpha-S-as-mapping-space}
  W_\alpha(S) = \{ s \colon E_\alpha(S)\to [0,\infty] \mid s(i_\inverts) + s(i_\outverts) = t_\alpha(i) \},
\end{equation}
and $W_\alpha(S) \to W_\alpha(S')$ for a map $S' \to S$ in $\trees_\alpha$ is given by push forward of functions along the map of finite sets $E_\alpha(S) \hookrightarrow E_\alpha(S')$.  There is no condition on $s(\epsilon_\smallbullet)$, the length of the edge of color $\h_\smallbullet$.  The isomorphism $W_\alpha(S_{i,1}) \to [0,\infty]^2$ sends $s$ to $(s(\epsilon_\smallbullet),r(s(i_\outverts), s(i_\inverts)))$ where
\[
	r(s_o, s_i) =  \frac{1-e^{-s_o}}{1 - e^{-s_i}},
\]
which lands in $[0,\infty]$ because $s_o + s_i = t_\alpha>0$. Note that $s_o=0$ if and only if $r(s_o,s_i) = 0$ and $s_o = t_\alpha$ if and only if $r(s_o, s_i)=\infty$.  Since the morphism $S_{i_1,1} \to S_{j,1}$ from diagram~\ref{diagram:swisscheese-alpha-double-cube} collapses the edge $(i_1)_\outverts$ we get $W_\alpha(S_{j,1}) \cong \{ (r_\smallbullet,r) \in W_\alpha(S_{i_1,1}) \mid  r=0\}$.  In the same diagram, the morphism $S_{i_2,1}\to S_{j,1}$ collapses the edge $(i_2)_\inverts$, so we have $W_\alpha(S_{j,1}) \cong \{ (r_\smallbullet,r) \in W_\alpha(S_{i_2,1}) \mid r=\infty \}$.  The unique morphism $S_{i,1} \to S_{i,0}$ collapses the edge $i_{\smallbullet}$ so that $W_\alpha(S_{i,0}) \cong \{ (r_\smallbullet,r) \in W_\alpha(S_{i,1})\mid r_\smallbullet = 0 \}$.  The rest can be deduced from these cases.

\begin{lemma}
\label{lemma:W-alpha-hocolim}
 For any functor $F \colon \trees_\alpha \to \topl$ the coend $W_\alpha \otimes_{\trees_{\alpha}} F$ is the homotopy colimit of $F$ over $\trees_{\alpha}$.
\end{lemma}
\begin{proof}
It is clear from diagrams~\ref{diagram:W-alpha-double-cube} and \ref{diagram:swisscheese-alpha-double-cube} that $W_\alpha(S)$ is the geometric realization of the nerve of the under category of $S$ for each object $S \in \trees_\alpha$.  In addition the maps $W_\alpha(S) \to W_\alpha(S')$ for $S' \to S$ agree with the maps obtained from the nerves of under categories.
%By pasting such rectangles together along the vertices of $T_\alpha$ we obtain the whole category $\trees_\alpha$.  Since we took $T_\alpha$ to be a tree with no edges of length 0, we find that $W_\alpha(\epsilon_1)$ is homeomorphic to the square $[0,\infty]^2$.  Every other object in the diagram is carried into a subspace of $W_\alpha(\epsilon_1)$.
%\begin{gather*}
%W_\alpha(v'_1)=0\times [0,\infty] \quad W_\alpha(v_1)= \infty\times [0,\infty]  \\
%W_\alpha(v_0') = (0,0) \quad W_\alpha(\epsilon_0) = [0,\infty]\times 0 \quad W_\alpha(v_0) = (\infty,0)
%\end{gather*}
%It is clear that $W_\alpha$ is naturally homeomorphic to the functor
%\[
% N(x) = \abs{ \nerve(\trees_{\alpha})_{x/}}
%\]
%where $(\trees_{\alpha})_{x/}$ is the under category of $x$.
\end{proof}

\begin{lemma}
  \label{lemma:p-inv-alpha-is}
  We can explicitly compute $p\inv(\iota\alpha)$ as
  \[
  	p\inv(\iota\alpha) \simeq \abs{\swcheese(T_\alpha)\tilde\alpha} \simeq (S^{d-1})^{\vee n},
  \]
  where $\swcheese(T_\alpha)\tilde\alpha$ is the composition of all vertex labels from $\iota\alpha$.
\end{lemma}
\begin{proof}
   Let $\trees_{\alpha,0}$ denote the full subcategory of $\trees_\alpha$ consisting of objects $S_{j,0}$ and $S_{i,0}$ for internal vertices $j$ and internal edges $i$.  This category is homotopy terminal, so by lemma~\ref{lemma:W-alpha-hocolim} and lemma~\ref{lemma:p-inv-alpha-coend} we have $p\inv(\iota\alpha) = \hocolim_{\trees_{\alpha,0}} \swisscheese_\alpha$.  This is the same as the homotopy colimit of the coequalizer diagram
     \[
     	\coprod_{i \in E(T_\alpha)} \abs{\id_\h} \rightrightarrows \coprod_{j\in V(T_\alpha)} \abs{\tilde\alpha(j)},
     \]
   	 where one arrow is given by including into output parts of the boundaries of $\abs{\tilde\alpha(j)}$'s, and the other arrow is given by including into input boundaries.  These maps are cofibrations with disjoint images.  Each space in the coequalizer diagram is cofibrant.  Thus the coequalizer diagram is already cofibrant as a functor $(\cdot \rightrightarrows \cdot) \to \topl$.  Thus we can compute the normal colimit.  It is clear that this is the same as composing the $\tilde\alpha(j)$'s via $T_\alpha$ then taking the realization of the result.  In addition $\abs{\beta}$ is equivalent to a wedge of $n$ spheres of dimension $d-1$ if $\beta \in \swcheese^{\h}(n,m)$.
\end{proof}

\begin{definition}
  \label{definition:W-alpha-1}
  Let $\trees_{\alpha,1}$ denote the full subcategory of $\trees_\alpha$ where we discard the objects $S_{j,0}$ and $S_{i,0}$ for $j \in J_{z_\alpha}$ and $i \in E(T_\alpha)$.  Define a functor $W_{\alpha,1} \colon \trees_{\alpha,1}^{op}\to \topl$ by setting
  \[
  	W_{\alpha,1}(S) = \{ s \colon E_\alpha(S) \to [0,\infty] \mid \sum_{i\in E_\alpha(S)} s(i) = \infty \}.
  \]
  %$W_{\alpha,1}(S_{j,1}) = \ast$ and $W_{\alpha,1}(S_{i,1}) = \{ (r,r') \in W_\alpha(S_{i,1}) \mid r+r'=\infty \}$ if $i \not\in \rt(T_\alpha)$ and $W_{\alpha,1}(S_{i,1}) = \{ (r,r') \in W_\alpha(S_{i,1}) \mid r =0 \mbox{ or } r'=\infty \}$.
\end{definition}

\begin{lemma}
  \label{lemma:p1-inv-alpha-as-a-coend}
  Suppose $t_\alpha < \infty$, and $n = 1$, then $p_1\inv(\alpha)$ is given by the coend
  \[
  	W_\alpha \otimes_{\trees_{\alpha,1}} \swisscheese_\alpha,
  \]
  where $\swisscheese_\alpha$ is the functor in definition~\ref{definition:W-alpha} restricted to $\trees_{\alpha,1}$ and $W_{\alpha,1}$ is defined in \ref{definition:W-alpha-1}.
\end{lemma}
\begin{proof}
  Let $\gamma\in F(\schonebullet)(1,m)$ such that $p_1(\gamma) = \alpha$.  Pick a representative $(f,g,s, \tilde\gamma)$ where $f \colon y \to (1,0|1,m)$, $y\leq 1$, $g\colon z \to y$, $s\in W(g)$ and $\tilde\gamma\in\swcheesebullet(z)$.  Consider  $\iota\gamma \in \schbullet(1,m)$, which is represented by $(fg,W_\infty(f)s, \tilde\gamma)$.  Recall that the condition $y \leq 1$  means that each connected component of the young forest $y$ has at most one input whose color lives in $\{ \e, \e_\smallbullet \}$.  This implies that $f$ has at least one internal edge $i \in E(f)$.  Thus $W_\infty(f)s(i) =\infty$ when $i$ is viewed as an internal edge in $fg$.

  We know $p\iota\gamma = \iota\alpha$, so $\iota\gamma$ is represented by some triple $(S,s',\tilde\gamma)$ with $S \in \trees_\alpha$, $s' \in W_\alpha(S)$, and $\tilde\gamma\in \swisscheese_\alpha(S)$.  The relations in $\schbullet$ preserve edges of length $\infty$, so we must have $s'(i) = \infty$ for some $i \in E(S)$.  We are assuming $t_\alpha (i) < \infty$ for all $i \in E(T_\alpha)$, so the infinite edge in $S$ must be in $E_\alpha(S)$.  This implies $s' \in W_{\alpha,1}(S)$.  Moreover we cannot have such an infinite edge if $S = S_{j,0}$ for some vertex $j$ or $S=S_{i,0}$ for some internal edge $i$.  Thus $S \in \trees_{\alpha,1}$.  This defines the map from $p\inv(\alpha)$ to the coend.  We leave the remainder to the reader.
\end{proof}

\begin{lemma}
\label{lemma:W-alpha-1-hocolim}
 For any functor $F \colon \trees_{\alpha,1} \to \topl$ the coend $W_{\alpha,1} \otimes_{\trees_{\alpha,1}} F$ is the homotopy colimit of $F$ over $\trees_{\alpha,1}$.
\end{lemma}
\begin{proof}
  The argument here is similar to the proof of lemma~\ref{lemma:W-alpha-hocolim}.
\end{proof}

\begin{corollary}
  \label{corollary:p1-inv-alpha-degree-1-is}
  If $t_\alpha < 0$ and $n =1$, then the fiber $p_1\inv(\alpha)$ is equivalent to $\partial_\h\abs{\swcheese(T_\alpha)\tilde\alpha} \simeq S^{d-1}$.
\end{corollary}
\begin{proof}
  By the same argument as in lemma~\ref{lemma:p-inv-alpha-is}, $\hocolim_{\trees_{\alpha,1}} \swisscheese_\alpha$ is equivalent to $\colim_{\trees_{\alpha,1}} \swisscheese_\alpha$.  This is easily computed as the $\h$-colored boundary of the composite of $\tilde \alpha$.
\end{proof}

\begin{proof}[proof of {proposition~\ref{proposition:fiber-equivalence}}]
  Recall $\alpha$ is represented by $f \colon y \to (n,m)$, $y\leq 1$, $g \colon z_\alpha \to y$, $t \in W(g)$ and $\tilde \alpha \in \swcheese(z_\alpha)$.  By applying relations in $\schinf$ we may assume $0<t<\infty$.  We may think of $(g, t)$ as representing an element of $\schone(y)$ which we can write as $(\alpha(j))_{ j \in J_y}$.  If $\alpha(j) \in \sch(n_j,m_j)$ then $n_j \leq 1$.  Clearly $p_1\inv(\alpha(j)) \simeq p\inv(\alpha(j)) \simeq \ast$ when $n_j = 0$.  Since $t_{\alpha(j)} < \infty$ we can use corollary~\ref{corollary:p1-inv-alpha-degree-1-is} to conclude $p_1\inv(\alpha(j)) \simeq \partial_\h \abs{(\swcheese(g)(\alpha))(j)}$.  The fiber $p_1\inv(\alpha)$ is equal to the colimit of the diagram
    \[
    	\coprod_{i \in E(f)} \abs{1_\h} \rightrightarrows \coprod_{j \in V(f)} p_1\inv(\alpha(j)),
    \]
    where one arrow is given by $\abs{1_\h} \simeq \partial_i\abs{\alpha(y(i))} \to \partial_\h\abs{(\swcheese(g)(\alpha))(y(i))}$ and the other by $\abs{1_\h}\simeq \partial_{\rt}\abs{\alpha(f(i))} \to \partial_\h\abs{(\swcheese(g)(\alpha))(y(i))}$.  This colimit is clearly $(S^{d-1})^{\vee n} \simeq p\inv(\iota\alpha)$.
\end{proof}

\begin{figure}[htbp]
 \centering
\label{figure:p-inv-alpha}
\includegraphics[width=\textwidth]{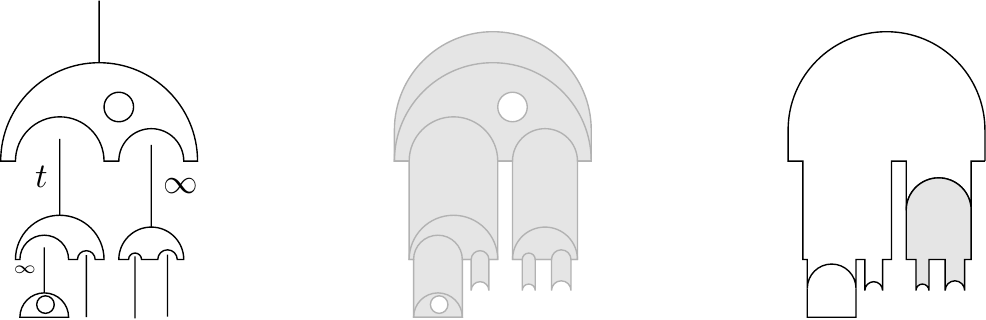}
\caption[Comparing $p_1\inv(\alpha)$ and $p\inv(\alpha)$]{On the left is
$\alpha \in \schinf(2,3)$.  In the middle is $p\inv(\alpha)$, and on the
right is $p_1\inv(\alpha)$.  Both $p\inv(\alpha)$ and $p\inv(\alpha)$ have
the homotopy type of a wedge of spheres, one for each disc in $\alpha$.}
\end{figure}


\begin{thebibliography}{Kon99}

\bibitem[BM06]{berger-moerdijk-w}
C.~Berger and I.~Moerdijk, \textsl{ The Boardman-Vogt resolution of operads in
  monoidal model categories},
\newblock Topology \textbf{ 45}(5), 807--849 (2006).

\bibitem[BV]{boardman-vogt-hiasots}
J.~Boardman and R.~Vogt, \textsl{ Homotopy Invariant Algebraic Structures on
  Topological Spaces, SpringerVerlag Lec},
\newblock Notes Math \textbf{ 347}.

\bibitem[Cos04]{costello-a-infinity}
K.~Costello, \textsl{ The A-infinity operad and the moduli space of curves},
\newblock Arxiv preprint math/0402015  (2004).

\bibitem[Get09]{getzler-operads-revisited}
E.~Getzler, \textsl{ Operads revisited},
\newblock Algebra, Arithmetic, and Geometry , 675--698 (2009).

\bibitem[Kon99]{kontsevich-operads-motives}
M.~Kontsevich, \textsl{ Operads and motives in deformation quantization},
\newblock Lett. Math. Phys. \textbf{ 48}(1), 35--72 (1999),
\newblock Mosh{\'e} Flato (1937--1998).

\bibitem[KS00]{kontsevich-soibelman-deformations}
M.~Kontsevich and Y.~Soibelman, \textsl{ Deformations of algebras over operads
  and the Deligne conjecture},
\newblock Math. Phys. Stud \textbf{ 21}, 255--307 (2000).

\bibitem[Tho12]{thomas-kontsevich-swiss-cheese}
J.~D. Thomas, \textsl{ Kontsevich's Swiss Cheese Conjecture},
\newblock (2012).

\end{thebibliography}
\end{document}